\documentclass[11pt]{article}
\usepackage{geometry,color}     
\usepackage[affil-it]{authblk}           
\usepackage{graphicx}
\usepackage{amssymb}
\usepackage{fullpage}
\usepackage{epstopdf}
\usepackage{multicol}
\usepackage{color}
\definecolor{clemson-orange}{RGB}{234,106,32}
\definecolor{chicago-maroon}{RGB}{128,0,0}
\definecolor{cincinnati-red}{RGB}{190,0,0}

\usepackage{natbib}
 \bibpunct[, ]{[}{]}{,}{n}{}{,}%

\usepackage{amsthm}

\usepackage[english]{babel}
\usepackage{amsmath}
\usepackage{enumerate}

\usepackage{epsfig}
\usepackage{subfigure}

\newcommand{\old}[1]{{}}

\newcommand{\bb}{\mathbb}

\newcommand{\spn}{\text{span}}

\newcommand{\R}{\bb R}

\newcommand{\N}{\bb N}

\DeclareMathOperator\cone{cone}

\def\Item$#1${\item $\displaystyle#1$
   \hfill\refstepcounter{equation}(\theequation)}

\theoremstyle{definition}
\newtheorem{prop}{Proposition}[section]
\newtheorem{theorem}[prop]{Theorem}
\newtheorem{lemma}[prop]{Lemma}

\newtheorem{definition}[prop]{Definition}
\newtheorem{remark}[prop]{Remark}

\newtheorem{example}[prop]{Example}

\newtheoremstyle{TheoremNum}
        {\topsep}{\topsep}              
        {\itshape}                      
        {}                              
        {\bfseries}                     
        {.}                             
        { }                             
        {\thmname{#1}\thmnote{ \bfseries #3}}
    \theoremstyle{TheoremNum}
    \newtheorem{theorem-pre}{Theorem}

\numberwithin{equation}{section}

\def\st{\mid}

\title{Strong duality and sensitivity analysis in semi-infinite linear programming}
\author{Amitabh Basu%
\thanks{\texttt{basu.amitabh@jhu.edu}}}
\affil{Department of Applied Mathematics and Statistics, Johns Hopkins University}

\author{Kipp Martin%
\thanks{\texttt{kmartin@chicagobooth.edu}} }

\author{Christopher Thomas Ryan%
\thanks{\texttt{chris.ryan@chicagobooth.edu}}}
\affil{Booth School of Business, University of Chicago}

\begin{document}

\maketitle

\noindent \textbf{Abstract:} Finite-dimensional linear programs satisfy strong duality (SD) and have the ``dual pricing" (DP) property.  The (DP) property ensures that, given a sufficiently small perturbation of the right-hand-side vector, there exists a dual solution that correctly ``prices" the perturbation by computing the exact change in the optimal objective function value. These properties may fail in semi-infinite linear programming where the constraint vector space is infinite dimensional. Unlike the finite-dimensional case, in semi-infinite linear programs the constraint vector space is a modeling choice.  We show that, for a sufficiently restricted vector space, both (SD) and (DP) always hold, at the cost of restricting the perturbations to that space. The main goal of the paper is to extend this restricted space to the largest possible constraint space where (SD) and (DP) hold. Once (SD) or (DP) fail for a given constraint space, then these conditions fail for all larger constraint spaces. We give sufficient conditions for when (SD) and (DP) hold in an extended constraint space.  Our results require the use of linear functionals that are singular or purely finitely additive and thus not representable as finite support vectors. The key to understanding these linear functionals is the extension of the Fourier-Motzkin elimination procedure to semi-infinite linear programs.

\smallskip



\section{Introduction}\label{s:introduction}

In this paper we examine how two standard properties of finite-dimensional linear 
programming, strong duality and sensitivity analysis, carry over to semi-infinite linear programs (SILPs). Our standard 
form for a semi-infinite linear program is
\begin{align*}\label{eq:SILP}
OV(b) := \inf \quad & \sum_{k=1}^n c_k x_k  \tag{SILP}\\
   {\rm s.t.} \quad & \sum_{k=1}^n a^k(i) x_k \ge b(i)  
   \quad\text{ for } i \in I
\end{align*}
where $a^k : I \to \R$ for all $k = 1, \dots, n$ and $b : I \to \R$ are 
real-valued functions on the (potentially infinite cardinality) index set 
$I$. The ``columns'' $a^k$ define a linear map $A : \R^n \to Y$ with 
$A(x) = (\sum_{k=1}^n a^k(i) x_k : i \in I)$ where $Y$ is a linear 
subspace of $\R^I$, the space of all real-valued functions on the index 
set $I$. The vector space $Y$ is called the \emph{constraint space} 
 of \eqref{eq:SILP}. This terminology follows  Chapter 2 of Anderson and Nash \cite{anderson-nash}. 
Goberna and L\'{o}pez \cite{goberna2014post} call $Y$ the ``space of parameters.'' 
Finite linear programming problem is a special case of  \eqref{eq:SILP} where 
 $I = \left\{1,\dots,m\right\}$ and $Y = \R^m$ for a finite natural 
number $m$.

As shown in Chapter~4 of Anderson and Nash \cite{anderson-nash}, 
the dual of \eqref{eq:SILP} with constraint space $Y$ is
\begin{align*}\label{eq:DSILPprime}
\begin{array}{rl}
\sup & \psi(b)  \\
  {\rm s.t.} & \psi(a^k) = c_k\quad \text{ for } k=1,\dots,n \\
& \psi \succeq_{Y'_+} 0
\end{array}\tag{\text{DSILP($Y$)}}
\end{align*}
where $\psi: Y \to \R$ is  a linear functional in the algebraic dual space $Y^{\prime}$ of $Y$ and $\succeq_{Y'_+}$ 
denotes an ordering of linear functionals induced by the cone 
\begin{align*}
Y'_+ := \left\{\psi : Y \to \R \st \psi(y) \ge 0 \text{ for all } 
y \in Y \cap \R^I_+\right\}
\end{align*}
where $\R^I_+$ is the set of all nonnegative 
real-valued functions with domain $I$.  The familiar finite-dimensional
linear programming 
dual has solutions $\psi =(\psi_1, \dots, \psi_m)$  where 
$\psi(y) = \sum_{i=1}^m y_i\psi_i$ for all nonnegative $y \in \R^m.$ Equivalently, $\psi\in \R^m_+$. Note the 
standard abuse of notation of letting $\psi$ denote both a linear 
functional and the real vector that represents it.

Our primary focus is on two desirable properties for the primal-dual 
pair \eqref{eq:SILP}--\eqref{eq:DSILPprime} when both the primal and 
dual are feasible (and hence the primal has bounded objective value).  
The first property is {\it strong duality} (SD).  
The primal-dual pair \eqref{eq:SILP}--\eqref{eq:DSILPprime} satisfies the 
\emph{strong duality} (SD) property if
\begin{itemize}
\item[] \textbf{(SD)}:  there exists a $\psi^{*} \in Y_{+}^{\prime}$ such that
\begin{align}\label{eq:strong-duality}
 \psi^*(a^k) = c_k \text{ for } k =1, 2,\dots n \text{ and } \psi^*(b) = OV(b)
\end{align}
\end{itemize}
where $OV(b)$ is the optimal value of the primal \eqref{eq:SILP} with 
right-hand-side $b.$ 

The second property of interest concerns use of dual solutions  
in sensitivity analysis. 
The primal-dual pair \eqref{eq:SILP}--\eqref{eq:DSILPprime} 
satisfies the \emph{dual pricing} (DP) property if
\begin{itemize}
\item[] \textbf{(DP)}: For every perturbation vector $d \in Y$  
such that \eqref{eq:SILP} is feasible for right-hand-side $b + d,$ there exists an optimal dual solution $\psi^{*}$ to 
\eqref{eq:DSILPprime} and an $\hat \epsilon > 0$ such that
\begin{align}\label{eq:dual-pricing}
OV(b + \epsilon d) = \psi^{*}(b + \epsilon d) = OV(b) + \epsilon \psi^{*}(d)
\end{align}
for all $\epsilon \in [0, \hat \epsilon].$
\end{itemize}

The terminology ``dual pricing'' refers to the fact that the 
appropriately chosen optimal dual solution $\psi^*$ correctly 
``prices'' the impact of changes in the right-hand on the optimal 
primal objective value. 

Finite-dimensional linear programs always satisfy (SD) and (DP) 
when the primal is feasible and bounded.  Define the vector space 
\begin{eqnarray}
U := \spn(a^1, \dots, a^n, b). \label{eq:define-U}
\end{eqnarray}
This is the minimum constraint space of interest  since the dual 
problem~\eqref{eq:DSILPprime}  requires the linear functionals 
defined  on $Y$ to operate on $a^1, \dots, a^n, b.$ 
If $I$ is a finite set and \eqref{eq:SILP} is feasible and 
bounded, then there exists a $\psi^{*} \in U^{\prime}_{+}$ 
such that~\eqref{eq:strong-duality}  and~\eqref{eq:dual-pricing} 
is satisfied.
Furthermore, optimal dual solutions $\psi^*$ that satisfy (SD) 
and (DP) are vectors in $\R^m.$ That is, we can take 
$\psi^{*} = (\psi^{*}_{1}, \ldots, \psi^{*}_{m}).$  Thus  $\psi^*$ 
is not only a  linear functional over $U,$  but it is also a linear functional over $\R^{m}.$  
The fact  that $\psi^*$ is a linear functional for  both $Y = U$ and 
$Y = \R^{m}$ is  obvious in the finite case and taken for granted.

The situation in semi-infinite linear programs is far more complicated and interesting. In general, a primal-dual pair \eqref{eq:SILP}--\eqref{eq:DSILPprime} can fail both (SD) and (DP).  Properties (SD) and (DP) depend crucially on the choice of constraint space $Y$ and its associated dual space. Unlike  finite linear programs where there is only one natural choice for the constraint space (namely $\R^m$), there are multiple viable nonisomorphic choices for an SILP. This makes constraint space choice a core modeling issue in semi-infinite linear programming. However, one of our main results is that (SD) and (DP) always hold with constraint space $U$. Under this choice, $\text{DSILP}(U)$ has a unique optimal dual solution $\psi^*$ we call the \emph{base dual solution} of \eqref{eq:SILP} -- see Theorem~\ref{theorem:silps-never-have-a-duality-gap}.  Throughout the paper, the linear functionals that are feasible 
to~\eqref{eq:DSILPprime}  are called dual solutions.

The base dual solution satisfies \eqref{eq:dual-pricing} for every choice of $d \in U$. However, this space greatly restricts the choice of perturbation vectors $d$. Expanding  $U$ to a larger space $Y$ (note that $Y$ must contain $U$ for \eqref{eq:DSILPprime} to be a valid dual) can compromise (SD) and (DP). 
We give concrete examples where (SD), (DP) (or both)  hold and do not hold.  

The  main tool used to extend $U$ to larger constraints spaces is the Fourier-Motzkin elimination procedure for semi-infinite linear programs introduced in Basu et al.~\cite{basu2013projection}.   We define  a  linear operator called the \emph{Fourier-Motzkin operator} that is used to map the constraint space   $U$ onto another constraint space. A linear functional is then defined on this new constraint space.  Under certain conditions, this linear functional is then extended using the Hahn-Banach theorem to a larger vector space that contains the new constraint space.  Then, using the adjoint of the Fourier-Motzkin operator, we get a linear functional on constraint spaces larger than $U$ where properties (SD) and (DP) hold.   
Although the Fourier-Motzkin elimination procedure described in  Basu et al. \cite{basu2013projection} was used  to study the finite support (or Haar) dual of an \eqref{eq:SILP},  this procedure provides insight into more general duals. The more general duals require the use of purely finitely additive linear functionals (often called {\em singular}) and these are known to be difficult to work with (see Ponstein, \cite{ponstein1981use}). However, the Fourier-Motzkin operator allows us to work with such functionals. 

\paragraph{Our Results.} 
Section~\ref{s:preliminaries} contains preliminary results on  constraint spaces and their duals. In Section~\ref{s:fm-elimination-duality} we recall some key results about the Fourier-Motzkin elimination procedure from Basu et. al.~\cite{basu2013projection} and also state and prove several additional lemmas that elucidate further insights into non-finite-support duals. Here we define the Fourier-Motzkin operator, which plays a key role in our theory.
In Section~\ref{s:strong-duality-dual-pricing} we prove (SD) and (DP) for the constraint space $Y = U.$  This is done in Theorems~\ref{theorem:silps-never-have-a-duality-gap} and~\ref{theorem:U-dual-pricing}, respectively. 

In Section~\ref{s:extending-U} we  prove (SD) and (DP) for subspaces $Y \subseteq \R^{I}$ that extend $U.$  In  Proposition~\ref{prop:mononticity} we show that once  (SD) or (DP) fail for a constraint space $Y$, then they fail for all larger constraint spaces.  Therefore, we want to extend the base dual solution and  push out from $U$ as far as possible until we encounter a constraint space for which (SD) or (DP) fail.  Sufficient conditions on the original data are provided that guarantee (SD) and (DP) hold in larger constraint spaces. See Theorems~\ref{theorem:extend-SD-Y}~and~\ref{theorem:sufficient-conditions-dual-pricing-alt}.

\paragraph{Comparison with prior work.} Our work can be contrasted with existing work on strong duality and sensitivity analysis in semi-infinite linear programs along several directions. First, the majority of work in semi-infinite linear programming  assumes either  the Haar dual or settings where $b$ and $a^k$ for all $k$ are continuous functions over a compact index set (see for instance Anderson and Nash \cite{anderson-nash}, Glashoff and Gustavson \cite{glashoff1983linear}, Hettich and Kortanek \cite{hettich1993semi}, and Shapiro \cite{shapiro2009semi}). The classical theory, initiated by Haar \cite{haar1924}, gave sufficient conditions for zero duality gap between the primal and the Haar dual. A sequence of papers by Charnes et al. \cite{charnes1963duality,charnes1965representations} and Duffin and Karlovitz \cite{duffin-karlovitz65}) fixed errors in Haar's original strong duality proof and described how a semi-infinite linear program with a duality gap could be reformulated to have zero duality gap with the Haar dual.  Glashoff in \cite{glashoff79}  also worked with a dual similar to the Haar dual.  The Haar dual was also used  during later development in the 1980s (in a series of papers by Karney \cite{karney81,karney1982pathological,karney85}) and remains the predominant setting for analysis in more recent work by Goberna and co-authors (see for instance, \cite{goberna2007sensitivity}, \cite{goberna1998linear} and \cite{goberna2014post}). 
By contrast, our work considers a wider spectrum of constraint spaces from $U$ to $\R^I$ and their associated algebraic duals. All such algebraic duals include the Haar dual (when restricted to the given constraint space), but also additional linear functionals. In particular, our theory handles settings where  the index set is not compact, such as $\N$. 

We do more than simply  extend the  Haar dual. Our work has a different focus and raises and answers questions not previously studied in the existing literature. We explore how \emph{changing} the constraint space (and hence the dual) effects duality and sensitivity analysis. This emphasis forces us to consider optimal dual solutions that are not finite support. Indeed, we provide examples where the finite support dual fails to satisfy (SD) but another choice of dual does satisfy (SD). In this direction, we extend our  earlier work in \cite{basu2014sufficiency} on the sufficiency of finite support duals to study semi-infinite linear programming through our use of the Fourier-Motzkin elimination technology. 

Second, our treatment of sensitivity analysis through exploration of the (DP) condition represents a different standard than the existing literature on that topic, which recently culminated in the monograph by Goberna and L\'{o}pez \cite{goberna2014post}. In (DP) we allow  a different dual solution in each perturbation direction $d$. The standard in Goberna and L\'{o}pez \cite{goberna2007sensitivity} and Goberna et al. \cite{goberna2010sensitivity} is that a single dual solution is valid for all feasible perturbations. This more exacting standard translates into strict sufficient conditions, including the existence of a primal optimal solution. By focusing on the weaker (DP), we are able to drop the requirement of primal solvability. Indeed, Example~\ref{ex:drop-primal-solvability} shows that (DP) holds even though a primal optimal solutions does not exist. Moreover, the sufficient conditions for sensitivity analysis in Goberna and L\'{o}pez \cite{goberna2007sensitivity} and Goberna et al. \cite{goberna2010sensitivity} rule out the possibility of dual solutions that are \emph{not} finite support yet nonetheless satisfy their standard of sensitivity analysis. Example~\ref{ex:drop-primal-solvability} provides one such case, where we show that there is a single optimal dual solution that satisfies \eqref{eq:dual-pricing} for all feasible perturbations $d$ and yet is not finite support.  

Third, the analytical approach to sensitivity analysis  in Goberna and L\'{o}pez \cite{goberna2014post} is grounded in convex-analytic methods that focus on topological properties of cones and epigraphs, whereas our approach uses Fourier-Motzkin elimination, an algebraic tool that   appeared in the study of semi-infinite linear programming duality in Basu et al. \cite{basu2013projection}. Earlier work by Goberna et al.~\cite{Goberna2010209} explored extensions of Fourier-Motzkin elimination to semi-infinite linear systems but did not explore its implications for duality.


\section{Preliminaries} \label{s:preliminaries}

In this section we review the notation,  terminology and properties of relevant constraint spaces and their algebraic duals used throughout the paper.

First some basic notation and terminology. The \emph{algebraic dual} $Y'$ of the vector space $Y$ is the set of real-valued linear functionals with domain $Y$. Let $\psi \in Y'.$ The evaluation of $\psi$ at $y$ is alternately denoted by $ \langle y, \psi \rangle$ or $\psi(y)$, depending on the context. 
 
A convex pointed cone $P$ in $Y$ defines a vector space ordering $\succeq_P$ of $Y$, with $y \succeq_P y'$ if  $y - y' \in P$. The \emph{algebraic dual cone} of $P$ is $P' = \left\{\psi \in Y' : \psi(y) \ge 0 \text{ for all } y \in P\right\}$.
Elements of $P'$ are called \emph{positive linear functionals} on $Y$ (see for instance, page 17 of Holmes \cite{holmes}). Let $A:X \rightarrow Y$ be a linear mapping from vector space $X$ to vector space  $Y$. The \emph{algebraic adjoint} $A' : Y' \to X'$ is a linear operator defined by $A'(\psi) = \psi \circ A$ where $\psi \in Y'$. 

We discuss some possibilities for the constraint space $Y$ in \eqref{eq:DSILPprime}.  A well-studied case is $Y = \R^I$. Here, the structure of \eqref{eq:DSILPprime} is complex since very little is known about the algebraic dual of $\R^I$ for general $I$. Researchers typically study an alternate dual called the \emph{finite support dual}. We denote the finite support dual of \eqref{eq:SILP} by
\begin{align*}\label{eq:FDSILP}
\begin{array}{rl}
\sup & \sum_{i = 1}^m \psi(i)b(i)  \\
  {\rm s.t.} & \sum_{i = 1}^m a^k(i)\psi(i) = c_k \quad \text{ for } k=1,\dots,n \\
& \psi \in \R^{(I)}_+
\end{array}\tag{\text{FDSILP}}
\end{align*}
where $\R^{(I)}$ consists of those functions in $\psi \in \R^I$ with $\psi(i) \neq 0$ for only finitely many $i \in I$ and $\R^{(I)}_+$ consists of those elements $\psi \in \R^{(I)}$ where $\psi(i) \ge 0$ for all $i \in I$.  A finite support  element of    $\R^{I}$   always represents a linear functional on any vector space  $Y \subseteq \R^{I}.$ Therefore the  finite support dual linear functionals feasible to \eqref{eq:FDSILP} are feasible to \eqref{eq:DSILPprime} for any constraint space $Y \subseteq \R^{I}$    that contains the space $U = \spn(a^1, \dots, a^n, b).$   This implies that  the optimal value of  $\eqref{eq:FDSILP}$ is always less than or equal to the optimal value of $\eqref{eq:DSILPprime}$ for all valid constraint spaces $Y$. It was shown in Basu et al. \cite{basu2014sufficiency} that \eqref{eq:FDSILP} and \eqref{eq:DSILPprime} for $Y = \R^\N$ are equivalent. In this case \eqref{eq:FDSILP} is indeed the algebraic dual of \eqref{eq:SILP} and so \eqref{eq:FDSILP} and $\text{DSILP}(\R^\N)$ are equivalent. This is not the necessarily the case for $Y = \R^I$ with $I \neq \N$.

Alternate choices for $Y$ include various subspaces of $\R^I$.  When $I = \N$
we pay particular attention to the spaces $\ell_{p}$ for    $1 \le p < \infty.$  The space  $\ell_{p}$
consist of all elements $y  \in  \R^N$   where $||y||_p = (\sum_{i \in I} |y(i)|^p)^{1/p} < \infty.$  
When $p = \infty$ we allow   $I$ to be uncountable and define   $\ell_{\infty}(I)$  to be the subspace of all  $y \in \R^{I}$  such that  $||y||_{\infty}= \sup_{i \in I} |y(i)| < \infty.$  
We also work with  the space $\mathfrak c$ consisting of all $y \in \R^\N$ where $\left\{y(i)\right\}_{i \in \N}$ is a convergent sequence and the space $\mathfrak c_0$ of all sequences convergent to $0$.

The spaces  $\mathfrak c$   and  $\ell_p$ for $1 \le p \le \infty$ defined above have special structure that is often used in examples in this paper. First, these spaces are Banach sublattices of $\R^\N$ (or  $\R^{I}$ in the case of $\ell_{\infty}(I)$) (see Chapter 9 of \cite{hitchhiker} for a precise definition). If $Y$ is a  Banach lattice, then  the positive linear functionals  in the algebraic dual $Y^{\prime}$ correspond exactly to the positive  linear functionals that are continuous in the norm topology on $Y$ that is used to define the Banach lattice. This follows from (a) Theorem~9.11 in Aliprantis and Border \cite{hitchhiker}, which shows that the norm dual $Y^*$ and the order dual $Y^\sim$ are equivalent in a Banach lattice and (b) Proposition~2.4 in Martin et al. \cite{martin-stern-ryan} that shows that the set of positive linear functionals in the algebraic dual  and the positive linear functionals in the order dual are identical.  This allows us to define $\text{DSILP}(\mathfrak c))$ and $\text{DSILP}(\ell_{p}))$ using the norm dual of $\mathfrak c$ and $\ell_{p},$ respectively.

For the constraint space $Y = \mathfrak c$ the linear functionals in its norm dual are characterized by 
\begin{align}\label{eq:define-c-functional}
\psi_{w \oplus r}(y) = \sum_{i=1}^\infty w_iy_i + ry_\infty
\end{align}
for all  $y \in \mathfrak c$ where $w \oplus r$ belong to $ \ell_1 \oplus \R$  and  $y_\infty = \lim_{i \to \infty} y_i \in \R$.  See Theorem 16.14 in Aliprantis and Border \cite{hitchhiker} for details.
This implies the positive linear functionals for $(\text{DSILP}(\mathfrak c))$ are isomorphic to vectors $w \oplus r \in (\ell_1)_+ \oplus \R_+$.   For obvious reasons, we call the linear functional $\psi_{0 \oplus 1}$ where $\psi_{0 \oplus 1} (y) = y_\infty$ the \emph{limit functional}. 

When $1 \le p < \infty$, the  linear functionals in the norm dual  are represented by sequences in the conjugate space $\ell_q$ with $1/p + 1/q = 1$. For $p = \infty$ and $I = \N,$  the   linear functionals $\psi$ in the norm  dual of $\ell_{\infty}(\N)$ can be expressed as $\psi = \ell_{1}  \oplus \ell_{1}^{d}$ where $\ell_{1}^{d}$ is the disjoint complement of $\ell_{1}$ and consists of all the singular linear functionals (see Chapter 8 of  Aliprantis and Border \cite{hitchhiker} for a definition of singular functionals). 
By  Theorem 16.31 in Aliprantis and Border \cite{hitchhiker}, for every functional $\psi \in \ell_1^d$ there exists some constant $r\in \R$ such that $\psi(y) =r  \lim_{i \to \infty} y(i)$ for $y \in \mathfrak c$.

\begin{remark}\label{rem:finite-OV}
If there is a $b$ such that $-\infty< OV(b) <\infty$ then  $-\infty< OV(0) <\infty$.   The first inequality follows from the fact that~\eqref{eq:SILP} is feasible and bounded for the given $b$ and the second inequality follows from feasibility of the zero solution.
Therefore, $OV(0)=0$ because in this case we are minimizing over a cone and we get a bounded value. 
\end{remark}


\section{Fourier-Motzkin elimination and its connection to duality}\label{s:fm-elimination-duality}

In this section we recall needed results from Basu et al. \cite{basu2013projection} on the Fourier-Motzkin elimination procedure for SILPs and the tight connection of this approach to the finite support dual. We also use the  Fourier-Motzkin elimination procedure to derive new results that  are applied to more general duals in later sections. 

To apply the Fourier-Motzkin elimination procedure we put \eqref{eq:SILP} into the ``standard'' form
\begin{eqnarray}  \qquad \inf z \phantom{-  c_{1} x_{1} -  c_{2} x_{2} - \cdots - c_{n} x_{n} + }   && \label{eq:initial-system-obj} \nonumber \\
\textrm{s.t.}  \quad z -  c_{1} x_{1} -  c_{2} x_{2} - \cdots - c_{n} x_{n}   &\ge&   0   \label{eq:initial-system-obj-con} \\
 \phantom{z + } a^{1}(i) x_{1} + a^{2}(i) x_{2} + \cdots + a^{n}(i) x_{n}  &\ge&  b(i)\quad \text{ for } i \in I.  \label{eq:initial-system-con}
\end{eqnarray}
The procedure takes \eqref{eq:initial-system-obj-con}-\eqref{eq:initial-system-con} as input and outputs the system
\begin{equation}\label{eq:J_system} 
\begin{array}{rcl}
\inf z && \\
0 &\ge&  \tilde{b}(h), \quad  h \in I_{1} \\
\phantom{z + } \tilde{a}^{\ell}(h) x_{\ell} + \tilde{a}^{\ell+1}(h) x_{\ell+1} + \cdots + \tilde{a}^{n}(h) x_{n}  &\ge& \tilde{b}(h), \quad h \in I_{2} \\
z    &\ge&   \tilde{b}(h), \quad  h \in I_{3} \\
z + \tilde{a}^{\ell}(h) x_{\ell} + \tilde{a}^{\ell+1}(h) x_{\ell+1} + \cdots + \tilde{a}_{n}(h) x_{n} &\ge& \tilde{b}(h), \quad h \in I_{4}
\end{array}
\end{equation}
where $I_1$, $I_2$, $I_3$ and $I_4$ are disjoint with $I_3 \cup I_4 \neq \emptyset$. Define $H := I_1 \cup \cdots \cup I_4$. The procedure also provides a set of finite support vectors $\{u^h\in \R^{(I)}_+: h \in H\}$ (each $u^h$ is associated with a constraint in~\eqref{eq:J_system}) such that $\tilde a^k(h) = \langle a^k, u^h \rangle$ for $\ell \le k \le n$ and $\tilde b(h) = \langle b, u^h \rangle.$  Moreover, for every $k=\ell, \ldots, n$, either $\tilde a^k(h) \geq 0$ for all $h \in I_2 \cup I_4$ or $\tilde a^k(h) \leq 0$ for all $h \in I_2 \cup I_4$. Further, for every $h\in I_2 \cup I_4$, $\sum_{k=\ell}^n |\tilde a^k(h)| > 0$.    Goberna et al.~\cite{Goberna2010209} also applied Fourier-Motzkin elimination to semi-infinite linear systems.  Their Theorem 5 corresponds to Theorem 2 in Basu et al. \cite{basu2013projection} and states that~\eqref{eq:J_system} is the projection of~\eqref{eq:initial-system-obj-con}-\eqref{eq:initial-system-con}.

The Fourier-Motzkin elimination procedure defines a linear operator called the Fourier-Motzkin operator and denoted $FM: \R^{\{0\}\cup I} \to \R^H$ where
\begin{eqnarray}
 FM(v) := (\langle v, u^h \rangle : h \in H) \textrm{ for all } v\in \R^{\{0\}\cup I}. \label{eq:define-FM}
\end{eqnarray}
The linearity of $FM$ is immediate from the linearity of $\langle \cdot , \cdot \rangle$. Observe that $FM$ is a positive operator since $u^h$ are nonnegative vectors in $\R^H$.  By construction,  $\tilde b = FM(0, b)$ and $\tilde a^k = FM((-c_k, a^k))$ for $k=1,\dots, n$.  

We  also  use the operator $\overline{FM} : \R^I \to \R^H$ defined by 
\begin{eqnarray}
\overline{FM}(y) := FM((0,y)).   \label{eq:define-FM-bar}
\end{eqnarray}
It is immediate from the properties of $FM$ that $\overline{FM}$ is also a positive linear operator.

\begin{remark}\label{rem:FM-op} 
See the description of the Fourier-Motzkin elimination procedure in  Basu et al. \cite{basu2013projection} and observe that the $FM$ operator does not change if we change $b$ in~\eqref{eq:SILP}. In what follows we assume a fixed $a^1, \ldots, a^n \in \R^I$ and $c \in \R^n$ and vary the right-hand-side $b$. This observation implies we have the same $FM$ operator for all SILPs with different right-hand-sides $y \in \R^I$. In particular,  the sets $I_1, \dots, I_4$ are the same for all right-hand-sides $y \in \R^I$. 
\end{remark}

The following basic lemma regarding the $FM$ operator is used throughout the paper.

\begin{lemma}\label{lemma:cute-little-trick}
For all $r \in \R$ and $y \in \R^I$,  $FM((r,y))(h) = r + FM((0,y))(h)$ for all $h \in I_3 \cup I_4$.
\end{lemma}
\begin{proof}
 By the linearity of the $FM$ operator $FM((r,y)) = r FM((1,0,0,\dots)) + FM((0,y)).$ If $h \in I_3 \cup I_4$ then $FM((1,0,0,\dots))(h) = 1$ because $(1,0,0,\dots)$ corresponds the $z$ column in \eqref{eq:initial-system-obj}-\eqref{eq:initial-system-con} and in \eqref{eq:J_system}, $z$ has a coefficient of $1$ for $h \in I_3 \cup I_4$. Hence, for $h \in I_3 \cup I_4$, $FM((r,y))(h) = r + FM((0,y))(h)$.
\end{proof}

Numerous properties of the primal-dual pair \eqref{eq:SILP}--\eqref{eq:FDSILP} are characterized in terms of the output system \eqref{eq:J_system}. The following functions play a key role in summarizing information encoded by this system.

\begin{definition}\label{def:S-L-OV-def}
Given a $y \in\R^I$, define $L(y) := \lim_{\delta \to \infty}\omega(\delta, y)$ where $\omega(\delta, y) := \sup \{\tilde{y}(h) - \delta \sum_{k=\ell}^{n} |\tilde{a}^k(h)| \, : \, h \in I_4 \}$, where $\tilde y = \overline{FM}(y)$.  Define $S(y) = \sup_{h \in I_3} \tilde y(h)$. 
\end{definition}

For any fixed $y\in \R^I$, $\omega(\delta, y)$ is a nonincreasing function in $\delta$. A key connection between the primal problem and these functions is given in Theorem~\ref{theorem:fm-primal-value}.

\begin{theorem}[Lemma 3 in Basu et  al. \cite{basu2013projection}]\label{theorem:fm-primal-value}
If  \eqref{eq:SILP} is  feasible then $OV(b) = \max\{S(b), L(b)\}.$
\end{theorem}

The following result describes useful properties of the functions $L$, $S$ and $OV$ that facilitate our approach to sensitivity analysis when perturbing the right-hand-side vector.

\begin{lemma}\label{lem:convex-L}
$L(y)$, $S(y)$, and $OV(y)$ are sublinear functions of $y \in \R^I$. 
\end{lemma}
\begin{proof} We first show the sublinearity of $L(y)$. For any $y, w \in \R^I$, denote $\tilde y = \overline{FM}(y)$ and $\tilde w = \overline{FM}(w)$. Thus $\overline{FM}(y + w) = \overline{FM}(y) + \overline{FM}(w) = \tilde y + \tilde w$ by the linearity of the $\overline{FM}$ operator. Observe that

$$\begin{array}{rcl}\omega(\delta, y+ w) &= &\sup \{\tilde{y}(h)+\tilde{w}(h) - \delta \sum_{k=\ell}^{n} |\tilde{a}^k(h)| \, : \, h \in I_4 \} \\
& = &  \sup \{(\tilde{y}(h) - \frac\delta2 \sum_{k=\ell}^{n} |\tilde{a}^k(h)|)+(\tilde{w}(h) - \frac\delta2  \sum_{k=\ell}^{n} |\tilde{a}^k(h)|) \, : \, h \in I_4 \} \\
& \leq & \sup \{(\tilde{y}(h) - \frac\delta2 \sum_{k=\ell}^{n} |\tilde{a}^k(h)|) \, : \, h \in I_4 \} + \sup\{(\tilde{w}(h) - \frac\delta2  \sum_{k=\ell}^{n} |\tilde{a}^k(h)|) \, : \, h \in I_4 \} \\
& = & \omega(\frac\delta2; y) + \omega(\frac\delta2;w)
\end{array}$$
Thus, $ L(y+w) = \lim_{\delta\to\infty}\omega(\delta, y+ w) \leq  \lim_{\delta\to\infty}\omega(\frac\delta2; y) + \lim_{\delta\to\infty}\omega(\frac\delta2; w) =  \lim_{\delta\to\infty}\omega(\delta, y) + \lim_{\delta\to\infty}\omega(\delta, w) = L(y) + L(w)$. This establishes the subadditivity of $L(y)$.  

Observe that for any $\lambda > 0$ and $y \in \R^I$, we have $\omega(\delta,\lambda y) = \lambda\omega(\frac\delta\lambda; y)$ and therefore $L(\lambda y) = \lim_{\delta\to\infty}\omega(\delta, \lambda y) = \lim_{\delta\to\infty}\lambda \omega(\frac\delta\lambda; y) = \lambda \lim_{\delta\to\infty} \omega(\frac\delta\lambda; y) = \lambda \lim_{\delta\to\infty} \omega(\delta, y) =\lambda L(y)$. This establishes the sublinearity of $L(y)$. 

We now show the sublinearity of $S(y)$. Let $y, w \in \R^I$, then
\begin{align*}
S(y + w) &= \sup \left\{\tilde y(h) + \tilde w(h) : h \in I_3 \right\} \\
         &\le \sup \left\{\tilde y(h) : h \in I_3 \right\} + \sup \left\{\tilde w(h) : h \in I_3 \right\} \\
         &= S(y) + S(w).
\end{align*}
For any $\lambda > 0$ we also have $S(\lambda y) = \lambda S(y)$ by the definition of supremum. This establishes that $S(y)$ is a sublinear function.

Finally, since $OV(y) = \max \left\{L(y), S(y)\right\}$ and $L(y)$ and $S(y)$ are sublinear functions, it is immediate that $OV(y)$ is sublinear. 
\end{proof}

The values $S(b)$ and $L(b)$ are used to characterize when \eqref{eq:SILP}--\eqref{eq:FDSILP} have zero duality gap.

\begin{theorem}[Theorem 13 in Basu et  al. \cite{basu2013projection}]\label{theorem:zero-duality-gap}  The optimal value of~\eqref{eq:SILP} is equal to the optimal value of~\eqref{eq:FDSILP}  if and only if (i) \eqref{eq:SILP} is feasible and (ii) $S(b) \ge L(b)$.
\end{theorem}

The next lemma is useful in cases where $L(b) > S(b)$ and hence (by Theorem~\ref{theorem:zero-duality-gap}) the finite support dual has a duality gap. A less general version of the result appeared as Lemma~7 in Basu et al. \cite{basu2013projection}.

\begin{lemma}\label{lem:seq-L}
Suppose $y \in \R^I$ and $\tilde y = \overline{FM}(y)$. If  $\{\tilde y(h_m) \}_{m\in \N}$ is any convergent sequence with indices  $h_{m} $ in $I_4$   such that $\lim_{m\to\infty} \sum_{k=\ell}^n |\tilde a^k(h_m)| \to 0,$  then $\lim_{m\to\infty} \tilde y(h_m) \leq L(y)$.  Furthermore, if $L(y)$ is finite,  there exists a sequence of distinct indices $h_m$ in $I_4$ such that $\lim_{m \to \infty}\tilde y(h_m) = L(y)$ and  $\lim_{m\to \infty}\tilde a^k(h_m) = 0$ for $k = 1, \ldots, n$.  \end{lemma}

\begin{proof}

We prove the  first part of the Lemma.  Let  $\{\tilde y(h_m) \}_{m\in \N}$ be  a convergent sequence with indices  $h_{m} $ in $I_4$      such that $\lim_{m\to\infty} \sum_{k=\ell}^n |\tilde a^k(h_m)| \to 0.$  We show  that $\lim_{m\to\infty} \tilde y(h_m) \leq L(y)$.  If $L(y) = \infty$ the result is immediate. Next assume $L(y) = -\infty.$  Since $\lim_{m\to\infty} \sum_{k=\ell}^n |\tilde a^k(h_m)| \to 0$, for every $\delta>0$, there exists $N_{\delta}\in \N$  such that for all $m \ge N_{\delta},$  $\sum_{k=\ell}^n |\tilde a^k(h_m)|  < \frac{1}{\delta}.$ Then
\begin{eqnarray*}
\omega(\delta, y) &=& \sup\{  \tilde{y}(h) -  \delta \sum_{k=\ell}^{n} |\tilde{a}^k(h)| \, : \, h \in I_{4} \} \\
&\ge& \sup\{\tilde{y}(h_{m}) - \delta  \sum_{k=\ell}^{n} |\tilde{a}^k(h_{m})| \, : \, m \in \N \}  \\
&\ge& \sup\{\tilde{y}(h_{m}) - \delta  \sum_{k=\ell}^{n} |\tilde{a}^k(h_{m})| \, : \, m \in \N, \, \, m \ge N_\delta \} \\
&\ge& \sup\{\tilde{y}(h_{m}) - \delta (\frac{1}{\delta}) \, : \, m \in \N, \, \, m \ge N_\delta \}  \\
&=&  \sup\{\tilde{y}(h_{m})  \, : \, m \in \N, \, \, m \ge N_\delta \}  - 1 \\
&\ge& \lim_{m \to \infty}\tilde y(h_m) -1.
\end{eqnarray*}
Therefore, $-\infty = L(y) = \lim_{\delta\to\infty}\omega(\delta,y) \geq \lim_{m \to \infty}\tilde y(h_m) -1$ which implies $\lim_{m \to \infty}\tilde y(h_m) = -\infty$.

Now consider the case where  $\{\tilde y(h_m) \}_{m \in \N}$  is a convergent sequence and $L(y)$ is finite.   
Therefore, if  we can  find a subsequence $\{\tilde y(h_{m_{p}}) \}_{p \in \N}$ of $\{\tilde y(h_m) \}_{m \in \N} $  such that  $\lim_{p \to\infty} \tilde y(h_{m_{p}}) \leq L(y)$ it follows that  $\lim_{m\to\infty} \tilde y(h_m) \leq L(y).$  Since $\lim_{\delta \to \infty}\omega(\delta, y) = L(y)$, there is a sequence $(\delta_p)_{p \in \N}$ such that $\delta_p \geq 0$ and $\omega(\delta_p,y) < L(y) + \frac{1}{p}$ for all $p \in \N$. Moreover, $\lim_{m \to \infty} \sum_{k=\ell}^n |\tilde{a}^k(h_m)| = 0$, implies that for every $p \in \N$ there is an $m_{p} \in \N$ such that for all $m \ge m_{p},$ $\delta_{p} \sum_{k = \ell}^{n} | \tilde{a}^{k}(h_{m})| < \frac{1}{p}.$ Thus, one can extract a subsequence $(h_{m_p})_{p\in \N}$ of $(h_m)_{m\in \N}$ such that $\delta_{p} \sum_{k = \ell}^{n} | \tilde{a}^{k}(h_{m_p})| < \frac{1}{p}$ for all $p \in \N.$ Then
\[
L(y) + \frac{1}{p} > \omega(\delta_{p},y) = \sup \{\tilde{y}(h) - \delta_p \sum_{k=\ell}^{n} |\tilde{a}^k(h)| \, : \, h \in I_4 \} \ge \tilde{y}(h_{m_{p}}) -  \delta_{p} \sum_{k = \ell}^{n} | \tilde{a}^{k}(h_{m_{p}})| > \tilde{y}(h_{m_{p}}) - \frac1p.
\]
Thus  $\tilde y(h_{m_p}) < L(y) + \frac2p$ which implies $\lim_{p\to\infty}\tilde y(h_{m_p}) \leq L(y)$. %
\bigskip

Now show the second part of the Lemma that if  $L(y)$ is  finite,  then  there exists a sequence of distinct indices $h_m$ in $I_4$ such that $\lim_{m \to \infty}\tilde y(h_m) = L(y)$ and $\lim_{m\to\infty} \sum_{k=\ell}^n |\tilde a^k(h_m)| = 0.$  By hypothesis,  $ \lim_{\delta \to \infty} \omega(\delta, y) = L(y) > -\infty$ so $I_4$ cannot be empty.  
Since $\omega(\delta, y)$ is a nonincreasing function of $\delta$, $\omega(\delta, y) \geq L(y)$ for all $\delta$. Therefore, $L(y) \leq \sup \{ \tilde{y}(h) -  \delta \sum_{k=\ell}^{n} |\tilde{a}^{k}(h)| \, : \, h \in I_4\}$ for every $\delta$. Define $\bar I := \{ h \in  I_4 \, : \, \tilde y(h) < L(y) \}$ and $\bar\omega(\delta, y) = \sup \{ \tilde{y}(h) -  \delta \sum_{k=\ell}^{n} |\tilde{a}^{k}(h)| \, : \, h \in I_4 \setminus \bar I\}$.  We consider two cases.

{\em \underline{Case 1: $\lim_{\delta \to\infty}\bar\omega(\delta, y) = -\infty.$}} Since $\lim_{\delta\to\infty }\omega(\delta, y) = L(y) > -\infty$ and both $\omega(\delta, y)$ and $\bar\omega(\delta, y)$ are nonincreasing functions in $\delta$, there exists a $\bar \delta \geq 0$ such that $\omega(\delta, y) \geq L(y) \geq \bar\omega(\delta, y)+1$ for all $\delta \geq \bar\delta$. Therefore, for all $\delta \geq \bar\delta$, $\omega(\delta, y) = \sup\{ \tilde{y}(h) -  \delta \sum_{k=\ell}^{n} |\tilde{a}^{k}(h)| \, : \, h \in I_4\} \geq L(y) > L(y)-1 \geq \bar\omega(\delta, y) = \sup\{ \tilde{y}(h) -  \delta \sum_{k=\ell}^{n} |\tilde{a}^{k}(h)| \, : \, h \in I_4\setminus \bar I\}$. This  strict gap implies that we can drop all indices in $I_4\setminus \bar I$ and obtain $\omega(\delta, y) = \sup\{ \tilde{y}(h) -  \delta \sum_{k=\ell}^{n} |\tilde{a}^{k}(h)| \, : \, h \in \bar I\}$ for all $\delta \geq \bar\delta$.

For every $m \in\N$, set $\delta_m = \bar\delta+ m$. Since $\delta_m \geq \bar\delta$, 
\begin{align*}
L(y) \leq \omega(\delta_m) =\sup \{ \tilde{y}(h) -  \delta _m \sum_{k=\ell}^{n} |\tilde{a}^{k}(h)| \, : \, h \in \bar I\}=\sup \{ \tilde{y}(h) -  (\bar\delta + m) \sum_{k=\ell}^{n} |\tilde{a}^{k}(h)| \, : \, h \in \bar I\}
\end{align*} 
and thus, there exists $h_m \in \bar I$ such that $L(y) - \frac{1}{m} < \tilde y(h_m) - (\bar\delta +m) \sum_{k=\ell}^{n} |\tilde{a}^{k}(h_m)| \leq \tilde y(h_m) - m \sum_{k=\ell}^{n} |\tilde{a}^{k}(h_m)|$. Since $\tilde y(h) < L(y)$ for all $h \in \bar I$, we have $$\begin{array}{rl}& L(y)-\frac{1}{m} < L(y) - m\sum_{k=\ell}^{n} |\tilde{a}^{k}(h_m)| \\ \Rightarrow & \sum_{k=\ell}^{n} |\tilde{a}^{k}(h_m)| < \frac{1}{m^2}.\end{array}$$ 

This shows that $\lim_{m \to \infty} \sum_{k=\ell}^n| \tilde{a}^k(h_m)| = 0$ which in turn implies that $\lim_{m\to\infty} \tilde a^k(h_m) = 0$ for all $k = \ell, \ldots, n$.   By definition of $I_{4},$ $\sum_{j=\ell}^n |\tilde{a}^{k}(h_m)| > 0$ for all $h_m \in \bar I \subseteq I_4$ so we can assume the indices $h_m$ are all distinct.  Also,

\begin{align*}
\begin{array}{rl}& L(y)-\frac{1}{m} < \tilde y(h_m) - m \sum_{k=\ell}^{n} |\tilde{a}^{k}(h_m)|  \\ \Rightarrow & L(y) - \frac{1}{m} < \tilde y(h_m)\end{array}
\end{align*}

Since $\tilde y(h_m) < L(y)$ (because $h_m \in \bar I$), we get $ L(y) - \frac{1}{m} < \tilde y(h_m) < L(y)$. And so $\lim_{m\to\infty} \tilde y(h_m) = L(y)$.

{\em \underline{Case 2: $\lim_{\delta \to\infty}\bar\omega(\delta, y) > -\infty.$}} Since $\omega(\delta, y) \geq \bar\omega(\delta, y)$ for all $\delta\geq 0$ and $\lim_{\delta\to\infty}\omega(\delta, y) = L(y) < \infty$, we have $-\infty < \lim_{\delta \to\infty}\bar\omega(\delta, y) \leq L(y) < \infty$. First we show that there exists a sequence of indices $h_m \in I_4 \setminus \bar I$ such that $\tilde a^k(h_m) \to 0$ for all $k = \ell, \ldots, n$. This is achieved by showing that $\inf\{\sum_{k=\ell}^n| \tilde{a}^k(h)| : h \in I_4\setminus \bar I\} = 0$. Suppose to the contrary that $\inf\{\sum_{k=\ell}^n| \tilde{a}^k(h)| : h \in I_4\setminus \bar I\} = \beta > 0$. Since $\bar\omega(\delta, y)$ is nonincreasing and $\lim_{\delta \to \infty}\bar\omega(\delta, y) < \infty$, there exists $\bar \delta \geq 0$ such that $\bar\omega(\bar\delta,y) < \infty$. Observe that $\lim_{\delta \to \infty}\bar\omega(\delta, y) = \lim_{\delta \to\infty} \bar\omega(\bar\delta + \delta,y)$. Then, for every $\delta \geq 0$, 
\begin{align*}
\bar\omega(\bar\delta + \delta,y) &= \sup \{ \tilde{y}(h) -  (\bar\delta + \delta) \sum_{k=\ell}^{n} | \tilde{a}^{k}(h)| \, : \, h \in I_4 \setminus \bar I \} \\
&= \sup \{ \tilde{y}(h) -  \bar\delta  \sum_{k=\ell}^{n} |\tilde{a}^{k}(h)|  -\delta \sum_{k=\ell}^{n} |\tilde{a}^{k}(h)| \, : \, h \in I_4\setminus \bar I\} \\
&\le \sup \{ \tilde{y}(h) -  \bar\delta  \sum_{k=\ell}^{n} |\tilde{a}^{k}(h)|  -\delta \beta \, : \, h \in I_4\setminus \bar I\} \\
&= \sup \{ \tilde{y}(h) -  \bar\delta \sum_{k=\ell}^{n} |\tilde{a}^{k}(h)| \, : \, h \in I_4\setminus \bar I\} - \delta\beta\\
 &= \bar\omega(\bar \delta,y) - \delta\beta.
\end{align*}
Therefore, $-\infty < \lim_{\delta\to\infty} \bar\omega(\bar \delta + \delta,y) \leq \lim_{\delta \to \infty} (\bar\omega(\bar\delta,y) - \delta\beta) = -\infty$, since $\beta > 0$ and $\bar\omega(\bar\delta,y) < \infty$.  This is a contradiction. Thus $0 = \beta = \inf\{\sum_{k=\ell}^n| \tilde{a}^k(h)| : h \in I_4\setminus \bar I\}$. Since $\sum_{k=\ell}^n| \tilde{a}^k(h)| > 0$ for all $h \in I_4$, there is a sequence of distinct indices $h_m \in I_4\setminus \bar I$ such that $\lim_{m \to \infty} \sum_{k=\ell}^n| \tilde{a}^k(h_m)| = 0$, which in turn implies that $\lim_{m\to\infty} \tilde a^k(h_m) = 0$ for all $k = \ell, \ldots, n$.  

Now we show there is a subsequence of  $\tilde y(h_m)$ that converges to $L(y)$.  Since $\lim_{\delta \to \infty}\bar\omega(\delta, y) \leq L(y)$, there is a sequence $(\delta_p)_{p \in \N}$ such that $\delta_p \geq 0$ and $\bar\omega(\delta_p,y) < L(y) + \frac{1}{p}$ for all $p \in \N$.   It was shown above that the sequence $h_m \in I_4\setminus \bar I$ is such that $\lim_{m \to \infty} \sum_{k=\ell}^n |\tilde{a}^k(h_m)| = 0$.  This implies that for every $p \in \N$ there is an $m_{p} \in \N$ such that for all $m \ge m_{p},$ $\delta_{p} \sum_{k = \ell}^{n} | \tilde{a}^{k}(h_{m})| < \frac{1}{p}.$ Thus, one can extract a subsequence $(h_{m_p})_{p\in \N}$ of $(h_m)_{m\in \N}$ such that $\delta_{p} \sum_{k = \ell}^{n} | \tilde{a}^{k}(h_{m_p})| < \frac{1}{p}$ for all $p \in \N.$ Then 
\[
L(y) + \frac{1}{p} > \bar\omega(\delta_{p},y) = \sup \{\tilde{y}(h) - \delta_p \sum_{k=\ell}^{n} |\tilde{a}^k(h)| \, : \, h \in I_4\setminus \bar I \} \ge \tilde{y}(h_{m_{p}}) -  \delta_{p} \sum_{k = \ell}^{n} | \tilde{a}^{k}(h_{m_{p}})| > \tilde{y}(h_{m_{p}}) - \frac1p.
\]

Recall that $h_{m_{p}} \in I_{4} \backslash \overline{I}$ implies $\tilde{y}(h_{m_{p}}) \ge L(y)$, and therefore
$
L(y) + \frac{2}{p} > \tilde{y}(h_{m_{p}}) \ge L(y).
$
By replacing $\{h_m\}_{m \in \N}$ by the subsequence $\{h_{m_p}\}_{p \in \N}$, we get $\tilde y(h_{m_{p}})$ as the desired subsequence that converges to $L(y)$. 

Hence, there exists is a sequence $\{h_m\}_{m\in \N}$ be any sequence of indices in $I_4$    such that $\tilde y(h_m) \to L(y)$ as $m \to \infty$ and $\tilde a^k(h_m) \to 0$ as $m \to \infty$ for $k = \ell, \dots, n$.  Also, $\tilde a^k(h_m) = 0$ for all $k=1, \ldots, \ell-1$.

\end{proof}

Although Lemma~\ref{lem:seq-S} and its proof are very simple (they essentially follow from the definition of supremum), we include it in order to be symmetric with Lemma~\ref{lem:seq-L}. Both results are needed for Proposition~\ref{prop:ov-seq}.

\begin{lemma}\label{lem:seq-S}
Suppose $y \in \R^I$ and $\tilde y = \overline{FM}(y)$ with $I_{3} \neq \emptyset.$     If  $\{\tilde y(h_m) \}_{m\in \N}$ is any convergent sequence with indices $ h_{m} $ in $I_3,$    then $\lim_{m\to\infty} \tilde y(h_m) \leq S(y)$.  Furthermore, there exists a sequence of distinct indices $h_m$ in $I_3$ such that $\lim_{m \to \infty}\tilde y(h_m) = S(y)$ and  $\lim_{m\to \infty}\tilde a^k(h_m) = 0$ for $k = 1, \ldots, n$.   Also, if  the supremum that defines $S(y)$ is not attained, the sequence of indices can be taken to be distinct. 

\end{lemma}

\begin{proof}
 By definition of supremum there exists a sequence $\{h_m\}_{m \in \N} \subseteq{I_3}$ such that $\tilde y(h_m) \to S(y)$ as $m \to \infty$. If the supremum that defines $S(y)$ is attained by $\tilde y(h_0) = S(y)$ then take $h_m = h_0$ for all $m \in \N$. Otherwise, the elements $h_m$ are taken to be distinct.   By definition of $I_{3},$   $\tilde a^k(h_m) = 0$ for $k = 1, \dots, n$  and   for all $m\in \N$ and so  $\lim_{m\to\infty} \sum_{k=\ell}^n |\tilde a^k(h_m)| = 0.$
 
 It also follows from the definition of supremum that if  $\{\tilde y(h_m) \}_{m\in \N}$ is any convergent sequence with indices $ h_{m} $ in $I_3,$    then $\lim_{m\to\infty} \tilde y(h_m) \leq S(y)$.
\end{proof}

\begin{prop}\label{prop:ov-seq}
Suppose $y \in \R^I$, $\tilde y = \overline{FM}(y)$ and $OV(y)$ is finite.  Then there exists a sequence of indices (not necessarily distinct) $h_m$ in $H$ such that $\lim_{m \to \infty}\tilde y(h_m) = OV(y)$ and  $\lim_{m\to \infty}\tilde a^k(h_m) = 0$ for $k = 1, \ldots, n$.   The sequence is contained entirely in $I_3$ or $I_4$. Moreover, either if $L(y) > S(y)$, or when $L(y) \le S(y)$ and the supremum that defines $S(y)$ is not attained, the sequence of indices can be taken to be distinct. 
\end{prop}

\begin{proof}
By Theorem~\ref{theorem:fm-primal-value},  $OV(y) = \max\{S(y), L(y)\}.$  The result is now immediate from Lemmas~\ref{lem:seq-L} and~\ref{lem:seq-S}.

\end{proof}


\section{Strong duality and dual pricing for a restricted constraint space}\label{s:strong-duality-dual-pricing}

Duality results for SILPs depend crucially on the choice of the constraint space $Y.$ In this section we work with the constraint space $Y = U$ where $U$ is defined in~\eqref{eq:define-U}.  Recall that the vector space $U$ is the minimum vector space of interest  since every legitimate dual problem~\eqref{eq:DSILPprime}  requires the linear functionals defined  on $Y$ to operate on $a^1, \dots, a^n, b.$ 
We show that when  $Y = U = \text{span}(a^1, \dots, a^n, b)$,  (SD) and (DP) hold.  In particular, we explicitly construct a linear functional $\psi^* \in U_{+}^{\prime}$ such that~\eqref{eq:strong-duality} and~\eqref{eq:dual-pricing} hold.

\begin{theorem}\label{theorem:silps-never-have-a-duality-gap}
Consider an instance of \eqref{eq:SILP} that is feasible and bounded. Then, the dual problem $(\text{DSILP}(U))$ with $U = \text{span}(a^1, \dots, a^n, b)$ is solvable and (SD) holds for the dual pair \eqref{eq:SILP}--$(\text{DSILP}(U)).$  Moreover, $(\text{DSILP}(U))$ has a unique optimal dual solution.
\end{theorem}

\begin{proof}

Since \eqref{eq:SILP} is feasible and bounded, we  apply Proposition~\ref{prop:ov-seq} with $y =b$ and extract a subset of indices $\{h_{m}\}_{m \in \N}$  of $H$ satisfying $\tilde b(h_m) \to OV(b)$ as $m \to \infty$ and $\tilde a^k(h_m) \to 0$ as $m \to \infty$ for $k = 1, \dots, n$. 
\vskip 5pt 

By Lemma~\ref{lemma:cute-little-trick}, for all $k=1,\ldots, n$, $\overline{FM}(a^k)(h_m) = FM((-c_k, a^k))(h_m) +c_k$ and therefore $\lim_{m\to\infty} \overline{FM}(a^k)(h_m) = \lim_{m\to\infty}FM((-c_k, a^k))(h_m) +c_k = \lim_{m\to\infty}\tilde a^k(h_m) +c_k = c_k$. Also, $\lim_{m\to\infty}\overline{FM}(b) = \lim_{m\to\infty}FM((0,b)) = \lim_{m\to\infty} \tilde b(h_m) = OV(b)$.  Therefore  $\overline{FM}(a^1), \ldots \overline{FM}(a^k),$  $\overline{FM}(b)$ all lie in the subspace $M \subseteq \R^{H}$ defined by
\begin{eqnarray}
M :=  \{  \tilde{y} \in \R^H \, :\, \tilde{y}(h_m)_{m \in \N} \,\, \text{ converges }   \}. \label{eq:define-M}
\end{eqnarray}

Define a  positive  linear functional $\lambda$ on $M$ by
\begin{align}\label{eq:projected-system-linear-functional}
\lambda(\tilde y) = \lim_{m \to \infty} \tilde{y}(h_m).
\end{align}

Since $\overline{FM}(a^1), \ldots \overline{FM}(a^k), \overline{FM}(b) \in M$ we have $\overline{FM}(U) \subseteq M$ and so $\lambda$ is defined on $\overline{FM}(U)$. Now map $\lambda$ to a linear functional in $U'$ through the adjoint mapping $\overline{FM}'$. Let $\psi^{*} = \overline{FM}'(\lambda).$  
We verify that $\psi^{*}$  is  an optimal solution to $(\text{DSILP}(Y))$ with objective value $OV(b)$. 

It follows from the definition of  $\lambda$ in~\eqref{eq:projected-system-linear-functional} that $\lambda$ is a positive linear functional.   Since $\overline{FM}$ is a positive operator, $\psi^{*}  = \overline{FM}'(\lambda) = \lambda\circ \overline{FM}$ is a positive linear functional on $U$.  We now check that $\psi^{*}$ is dual feasible. We showed above that $\lambda(\overline{FM}(a^{k}))  = c_{k} $ for all $k = 1, \ldots, n.$  Then by definition of adjoint
\begin{eqnarray*}
\langle a^k, \psi^{*} \rangle = \langle  a^k, \overline{FM}'(\lambda)\rangle \
      = \langle \overline{FM}(a^k), \lambda\rangle 
      = c_{k}.
\end{eqnarray*}

By a similar argument, $\langle  b, \psi^{*} \rangle = \langle  \overline{FM}(b), \lambda \rangle =   OV(b)$  so $\psi^{*}$ is both feasible and optimal.    Note that $\psi^{*}$ is the {\it unique} optimal dual solution since $U$ is the span of  $a^{1}, \ldots, a^{n}$ and $b$ and defining the value of $\psi^{*}$  for each of these vectors uniquely determines an optimal dual solution. This completes the proof.\end{proof}

\begin{remark}\label{rem:constrast-with-kortanek}
The above theorem can be contrasted with the results of Charnes at el. \cite{charnes1965representations} who proved that is always possible to reformulate \eqref{eq:SILP} to ensure zero duality gap with the finite support dual program. Our approach works with the original formulation of \eqref{eq:SILP} and thus preserves dual information in reference to the original system of constraints rather than a reformulation. Indeed, our procedure considers an alternate \emph{dual} rather than the finite support dual.  
\end{remark}

\begin{theorem}\label{theorem:U-dual-pricing}
Consider an instance of \eqref{eq:SILP} that is feasible and bounded. Then the unique optimal dual solution $\psi^{*}$ constructed in Theorem~\ref{theorem:silps-never-have-a-duality-gap} satisfies \eqref{eq:dual-pricing} for all perturbations $d \in U.$ 
\end{theorem}
\begin{proof}
By  hypothesis \eqref{eq:SILP} is feasible and bounded. Then by Theorem~\ref{theorem:silps-never-have-a-duality-gap}     there is an optimal dual solution $\psi^{*}$ such that $\psi^{*}(b) = OV(b).$  For now assume \eqref{eq:SILP} is also solvable with optimal solution $x(b).$ We relax this assumption later.

We show  {\it for every} perturbation $d \in U$ that $\psi^{*}$ is an optimal dual solution.    If  $d \in U$  then $d = \sum_{k=1}^n \alpha_k a^k + \alpha_0 b.$  Following the logic of Theorem~\ref{theorem:silps-never-have-a-duality-gap}, there exists a subsequence $\left\{h_m\right\}$ in $I_3$ or $I_4$ such that  $\tilde{a}^{k}(h_{m}) \rightarrow 0$  for $k = 1, \ldots, n$ and  $\tilde{b}(h_{m}) \rightarrow OV(b).$ Since a linear combination of convergent sequences is a convergent sequence the  linear functional $\lambda$ defined in~\eqref{eq:projected-system-linear-functional} is well defined   for   $\overline{FM}(U),$  and in particular for  $\overline{FM}(b + d).$ For the projected system~\eqref{eq:J_system},  $\lambda$ is dual feasible and gives objective function value 
\begin{eqnarray*}
\psi^{*}(b + d) =  \lambda(\overline{FM}(b+d)) =  (1 + \alpha_{0})OV(b)  + \sum_{k = 1}^{n} \alpha_{k} c_{k}. 
\end{eqnarray*}
A primal feasible solution to    \eqref{eq:SILP} with right-hand-side $b + d$  is   $\hat x_k = (1 +  \alpha_0) x_k(b) +  \alpha_k, \,\text{ for } k = 1, \dots, n$ and this primal solution gives objective function value $(1 + \alpha_{0})OV(b)  + \sum_{k = 1}^{n} \alpha_{k} c_{k}$. By weak duality $\psi^{*}$ remains the optimal dual solution for right-hand-side $b + d.$

Now consider the case where  \eqref{eq:SILP}  is not solvable.  In this case  the optimal  primal objective value is attained as a supremum.  In this case there is a  sequence $\{x^{m}(b)\}$     of primal feasible solutions whose objective function values converges  $OV(b).$

Now construct a sequence of feasible solutions $\{\hat{x}^{m}(b)\}$  using the  definition  of  $\hat x$ above. Then a very similar reasoning to the above shows that the sequence   $\{\hat{x}^{m}(b)\}$   converges to the value  $\psi^{*}(b + d).$  Again, by  weak duality $\psi^{*}$ remains the optimal dual solution for right-hand-side $b + d.$
\end{proof}

$(\text{DSILP}(U))$ is a very special dual. If there exists  a $b$ for which \eqref{eq:SILP}    is feasible and bounded, then  there is an optimal dual solution $\psi^{*}$  to   $(\text{DSILP}(U))$  such that
\begin{eqnarray*}
OV(b + d) = OV(b) + \psi^{*}(d)
\end{eqnarray*}
{\it for every} $d \in U.$ 
This is a much stronger result than (DP) since the same  linear functional $\psi^{*}$ is  valid for every perturbation $d.$  A natural question is when  the weaker property (DP) holds in  spaces that strictly contain $U$.  The problem of allowing perturbations  $d \not\in U$ is that  $\overline{FM}(d)$  may not lie in the subspace $M$ defined by~\eqref{eq:define-M} and therefore the  $\lambda$ defined in~\eqref{eq:projected-system-linear-functional}  is not defined for $\overline{FM}(d).$    Then we cannot use the adjoint operator $\overline{FM}'$  to get $\psi^{*}(d).$  This motivates the  development of the next section where we want to find the largest possible perturbation space so that (SD) and (DP) hold.


\section{Extending strong duality and dual pricing to larger constraint spaces}\label{s:extending-U}

 The goal of this section is to prove (SD) and (DP) for subspaces   $Y \subseteq \R^{I}$ that extend $U$.     In Proposition~\ref{prop:duality-gap-iff-lift-base-functional} below we prove that the primal-dual pair \eqref{eq:SILP}--\eqref{eq:DSILPprime} satisfy (SD) if and only if the base dual solution $\psi^*$ constructed in Theorem~\ref{theorem:silps-never-have-a-duality-gap} can be extended to a positive linear functional over $Y$. 

\begin{prop}\label{prop:duality-gap-iff-lift-base-functional}
Consider an instance of \eqref{eq:SILP} that is feasible and bounded and $Y$ a subspace of $\R^I$ that contains $U$ as a subspace. Then dual pair \eqref{eq:SILP}--\eqref{eq:DSILPprime} satisfies (SD) if and only if the base dual solution $\psi^*$ defined in \eqref{eq:strong-duality} can be extended to a positive linear functional over $Y$. 
\end{prop}
\begin{proof}
If $\psi$ is an optimal dual solution it must be feasible and thus $\psi(a^k) = c_k$ for $k = 1, \dots, n$ and $\psi(b) = OV(b)$. In other words, $\psi(y) = \psi^*(y)$ for $y \in U$. Thus, $\psi$ is a positive linear extension of $\psi^*$. Conversely, every positive linear extension $\psi$ of $\psi^*$ is dual feasible and satisfies $\psi(b) = OV(b)$. This is because any extension maintains the values of $\psi^*$ when restricted to $U$. 
\end{proof}
Moreover, we have the following ``monotonicity'' property of (SD) and (DP).

\begin{prop}\label{prop:mononticity}
Let $Y$ a subspace of $\R^I$ that contains $U$ as a subspace. Then 
\begin{enumerate}

\item  if the primal-dual pair \eqref{eq:SILP}--\eqref{eq:DSILPprime} satisfies (SD), then (SD) holds for every primal dual pair \eqref{eq:SILP}--$(\text{DSILP}(Q))$ where $Q$ is a subspace of $Y$ that contains $U$. 

\item if the primal-dual pair \eqref{eq:SILP}--\eqref{eq:DSILPprime} satisfies (DP), then (DP) holds for every primal dual pair \eqref{eq:SILP}--$(\text{DSILP}(Q))$ where $Q$ is a subspace of $Y$ that contains $U$.   
\end{enumerate}

\end{prop}
\begin{proof}

Property (DP) implies property (SD) so in both cases 1. and 2. above \eqref{eq:SILP}--\eqref{eq:DSILPprime} satisfies (SD).  Then by Proposition~\ref{prop:duality-gap-iff-lift-base-functional} the base dual solution $\psi^*$ defined in \eqref{eq:strong-duality} can be extended to a positive linear functional $\bar{\psi}$ over $Y$.  Since $Q \subset Y,$  $\bar{\psi}$  is defined on $Q$  and    is an optimal dual solution with respect to the space  $Q$ since $OV(b) = \psi^{*}(b) = \bar{\psi}(b)$ and part 1. is proved.

Now show part 2.   Assume there is a $d \in Q \subseteq Y$ and  $b +  d$  is a feasible    
right-hand-side to~\eqref{eq:SILP}.  By definition of (DP) there there is an $\hat{\epsilon} > 0$ such that
\begin{eqnarray*}
OV(b + \epsilon d) = \bar{\psi}(b + \epsilon d) = OV(b) + \epsilon \bar{\psi}(d)
\end{eqnarray*}
holds for all $\epsilon \in [0, \hat \epsilon].$  But $Q \subset Y$ implies $\bar{\psi}$ is the optimal linear functional with respect to the constraint space $Q$ and property (DP) holds. \end{proof}

Another view of Propositions~\ref{prop:duality-gap-iff-lift-base-functional}~and~\ref{prop:mononticity}  is that once properties  (SD) or (DP) fail for a constraint space $Y$, then these properties fail for all larger constraint spaces. As the following example illustrates, an inability to extend can happen almost immediately as we enlarge the constraint space from $U.$ 

\begin{example}\label{ex:duality-gap-cannot-lift}
Consider the \eqref{eq:SILP}
\begin{eqnarray}\label{eq:cannot-lift-system}
\min x_{1} && \\
(1/i) x_{1} + (1/i)^{2} x_{2} &\ge& (1/i), \quad i \in \N. 
\end{eqnarray}
The smallest of the $\ell_{p}(\N)$ spaces that contains the columns of \eqref{eq:cannot-lift-system} (and thus $U$) is $Y = \ell_2$. Indeed, the first column is not in $\ell_1$ since $\sum_{i} \frac{1}{i}$ is not summable. We show (SD) fails to hold under this choice of $Y =  \ell_{2}.$ This implies that (DP) fails in $\ell_{2}$ and every space that contains $\ell_{2}.$

An optimal primal solution is  $x_{1} = 1$ and $x_{2} = 0$ with optimal solution value $1$.  The dual $\text{DSILP}(\ell_2)$ is
\begin{align}
\sup \quad &  \sum_{i=1}^\infty \frac{\psi_i}{i} \nonumber  \\
  {\rm s.t.} \quad & \sum_{i=1}^\infty \frac{\psi_i}{i} = 1 \label{first-constraint} \\
             & \sum_{i=1}^\infty \frac{\psi_i}{i^2} = 0 \label{second-constraint} \\
& \psi \in (\ell_2)_{+}. \nonumber
\end{align}
In writing $\text{DSILP}(\ell_2)$   we use  the fact that $(\ell^{\prime}_2)_+$ is isomorphic to $(\ell_2)_{+}$ (see the discussion in Section~\ref{s:preliminaries}). Observe that no nonnegative $\psi$ exists that can satisfy both \eqref{first-constraint} and \eqref{second-constraint}. Indeed, \eqref{second-constraint} implies $\psi_i = 0$ for all $i \in \N$. However, this implies that \eqref{first-constraint} cannot be satisfied. Hence, $\text{DSILP}(\ell_2) = -\infty$ and there is an infinite duality gap. Therefore (SD) fails, immediately implying that (DP) fails.
\quad $\triangleleft$
\end{example}

\paragraph{Roadmap for extensions.} Our goal is to provide a coherent theory of when  properties  (SD) and (DP) hold in spaces larger than $U.$ Our approach is to extend the base dual solution to larger spaces using Fourier-Motzkin machinery.  
We provide a brief intuition for the method, which is elaborated upon carefully in the proofs that follow. First, the Fourier-Motzkin operator $\overline{FM}(y)$ defined in~\eqref{eq:define-FM-bar} is used to map $U$ onto the vector space $\overline{FM}(U).$ 
Next a linear functional $\lambda(\tilde{y})$ (see~\eqref{eq:projected-system-linear-functional})  is defined over $\overline{FM}(U)$.  We aim to extend this linear functional to a larger vector space. Define the set
\begin{eqnarray}
\hat Y := \{y \in Y: -\infty < OV(y) < \infty\} \label{eq:define-E-hat}.
\end{eqnarray}
Note that $\hat Y$ is the set of ``interesting'' right hand sides, so it is a natural set to investigate. Extending to all of $Y$ beyond $\hat Y$ is unnecessary because these correspond  to right hand sides which give infeasible or unbounded primals.
However, the set $\hat Y$  is not necessarily a vector space, which makes it hard to talk of dual solutions acting on this set.  
If $\hat Y$ is a vector space, then $\overline{FM}(\hat{Y})$  is also a vector space and we show it is valid under the hypotheses of the Hahn-Banach Theorem to extend the linear functional $\lambda$ defined in~\eqref{eq:projected-system-linear-functional}  from   $\overline{FM}(U)$ to $\bar{\lambda}$ on $\overline{FM}(\hat{Y}).$   Finally, the adjoint $\overline{FM}'$ of the Fourier-Motzkin operator $\overline{FM}$ is used to map the extended linear functional $\bar{\lambda}$ to an optimal linear functional on  $\hat{Y}.$  Under appropriate conditions detailed below, this allows us to work with constraint spaces $\hat{Y}$  that strictly contain $U$ and still satisfy (SD) and (DP).   See Theorems~\ref{thm:SD-ell-infty}  and~\ref{theorem:sufficient-conditions-dual-pricing-alt} for careful statements and complete details. Figure~\ref{figure:extend-SD-Y} may help the reader keep track of the spaces involved.  We emphasize that in order for $(\text{DSILP}(\hat{Y}))$ to be well defined, $\hat{Y}$ must contain $U$ and itself be a vector space. 

\begin{figure}[ht]
\centering
\resizebox{3.5in}{!}{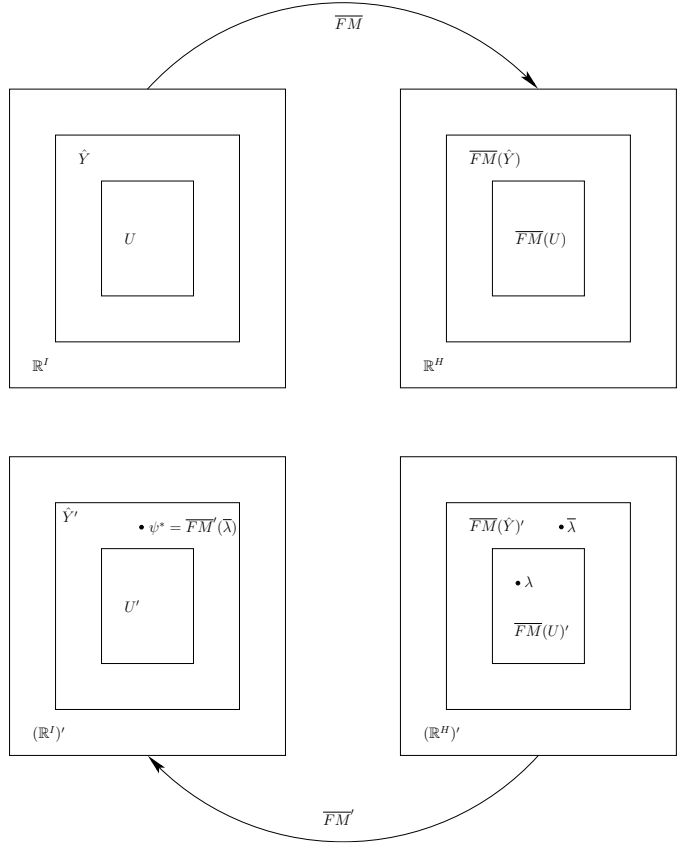}
\vskip 10pt
\caption{Illustrating Theorem~\ref{theorem:extend-SD-Y}.}\label{figure:extend-SD-Y}
\end{figure}

\subsection{Strong duality for extended constraint spaces}

The following lemma is used to show $U \subseteq \hat Y$ in the subsequent discussion. 

\begin{lemma}\label{lemma:ov-ak-bounded-finite} 
If $-\infty < OV(b) < \infty$ (equivalently,  \eqref{eq:SILP} with right-hand-side $b$ is feasible and bounded), then $-\infty <OV(a^k)< \infty$ for all $k=1, \ldots, n$.
\end{lemma}

\begin{proof}
If the right-hand-side vector is $a^{k}$ then $x_{k} = 1$ and $x_{j}  = 0$ for $j \neq k$ for a feasible objective value $c_{k}.$  Thus  $OV(a^k) \le  c_{k} < \infty.$

Now show $OV(a^{k}) > -\infty.$  Since $OV(a^k)  < \infty,$ by Lemma~\ref{theorem:fm-primal-value}, $OV(a^{k}) = \max\{ S(a^{k}),  L(a^{k})   \}.$  If $I_{3} \neq \emptyset$ then $S(a^{k})   > -\infty$ which implies $OV(a^{k}) > -\infty$ and we are done.  Therefore assume $I_{3} = \emptyset.$ Then $S(b) = -\infty.$  However, by  hypothesis $-\infty < OV(b) < \infty$ so by Lemma~\ref{theorem:fm-primal-value}  
\begin{align*}
OV(b) = \max\{ S(b),  L(b)   \} =   \max\{ -\infty,  L(b)   \}
\end{align*}
which implies $-\infty < L(b) < \infty.$  Then by Lemma~\ref{lem:seq-L}  there exists a sequence of distinct indices $h_m$ in $I_4$ such that  $\lim_{m\to \infty}\tilde a^k(h_m) = 0$ for all $k = \ell, \ldots, n.$    Note also that $\tilde{a}^{k}(h) = 0$ for $k = 1, \ldots, \ell -1$ and $h \in I_{4}.$  Let $\tilde{y} = \overline{FM}(a^{k}).$ Then  $\lim_{m\to \infty}\tilde a^k(h_m) = 0$ implies  by Lemma~\ref{lemma:cute-little-trick}, $\lim_{m\to \infty} \tilde{y}(h_m) = c_{k}.$  Again by Lemma~\ref{lem:seq-L}, $L(a^{k}) \ge \lim_{m \rightarrow \infty} \tilde{y}(h_{m}) = c_{k}.$
\end{proof}

\begin{theorem}\label{theorem:extend-SD-Y}
Consider an instance of \eqref{eq:SILP} that is feasible and bounded.  Let  $Y$   be a  subspace of  $ \R^I$   such that  $U \subset Y$ and  $\hat{Y}$ is a vector space.      Then  the dual problem $(\text{DSILP}(\hat{Y}))$ is solvable and (SD) holds for the primal-dual pair \eqref{eq:SILP}--$(\text{DSILP}(\hat{Y})).$  
\end{theorem}

\begin{proof}  The proof of this theorem is similar to the proof of Theorem~\ref{theorem:silps-never-have-a-duality-gap}. We  use the operator $\overline{FM}$ and consider the linear functional $\lambda$ defined in~\eqref{eq:projected-system-linear-functional} which was shown to be a linear functional on $\overline{FM}(U).$ By hypothesis,  $U \subset Y$ and so by Lemma~\ref{lemma:ov-ak-bounded-finite},   $U \subseteq \hat{Y}$ which implies  $\overline{FM}(U) \subseteq \overline{FM}(\hat{Y})$.     Since $\hat Y$ is a vector space,  $\overline{FM}(\hat{Y})$  is a vector space since $\overline{FM}$ is a linear operator.
We   use the Hahn-Banach theorem to extend $\lambda$ from $\overline{FM}(U)$ to $\overline{FM}(\hat{Y})$.  First observe that if $\overline{FM}(y^1) = \overline{FM}(y^2) = \tilde y$, then $S(y^1)=S(y^2)$ and $L(y^1) = L(y^2)$ because these values only depend on $\tilde y$, and therefore, $OV(y^1) = OV(y^2)$. This means for any $\tilde y \in \R^H$, $S,L$ and $OV$ are constant functions on the affine space $\overline{FM}^{-1}(\tilde y)$. Thus, we can push forward   the sublinear function $OV$ on $\hat{Y}$  by setting $p(\tilde y) = OV(\overline{FM}^{-1}(\tilde y))$ ($p$ is sublinear as it is the composition of the inverse of a  linear function  and a sublinear function). Moreover, by Lemmas~\ref{lem:seq-L}-\ref{lem:seq-S}  and Theorem~\ref{theorem:fm-primal-value}, $\lambda(\tilde y) \leq \max\{S(y), L(y)\} = OV(y) = p(\tilde y)$ for all $\tilde y  \in \overline{FM}(U)$. Then by the Hahn-Banach Theorem there exists an extension of $\lambda$ on $\overline{FM}(U)$  to $\bar{\lambda}$ on $\overline{FM}(\hat{Y})$ such that
\begin{align*}
-p(-\tilde y)  \le \bar{\lambda}(\tilde{y}) \le  p(\tilde y)
\end{align*}
for all $\tilde{y} \in \overline{FM}(\hat{Y}).$  We now show $\bar{\lambda}(\tilde{y})$ is positive on $\overline{FM}(\hat{Y}).$   If $\tilde{y} \ge 0$ then $-\tilde{y} \le 0$ and $\omega(\delta, -\tilde{y}) = \sup \{-\tilde{y}(h) - \delta \sum_{k=\ell}^{n} |\tilde{a}^k(h)| \, : \, h \in I_4 \} \le 0$  for all $\delta.$ Then $L(-y) = \lim_{\delta \rightarrow \infty}\omega(\delta, -\tilde{y}) \le 0$ for any $y$ such that $\tilde y  =\overline{FM}(y)$. Likewise  $S(-y)  = \sup \{-\tilde{y}(h)   \, : \, h \in I_3 \} \le 0.$  Then $S(-y), L(-y) \le 0$ implies
\begin{align*}
-p(-\tilde y) = -OV(-y) =  -\max\{ S(-y), L(-y) \}  = \min\{ -S(-y), -L(-y) \} \ge 0
\end{align*}
and $-p(-\tilde{y})  \le \bar{\lambda}(\tilde{y})$ gives $0 \le \bar{\lambda}(\tilde{y})$ on $\overline{FM}(\hat{Y}).$
 
 We have shown that $\bar{\lambda}$ is a positive linear functional on $\overline{FM}(\hat{Y}).$  It follows that  $\psi^{*} = \overline{FM}^{\prime}(\bar{\lambda})$ is a positive linear functional on $\hat{Y}.$
 
Now recall that the $\lambda$ defined in~\eqref{eq:projected-system-linear-functional} in Theorem~\ref{theorem:silps-never-have-a-duality-gap} had the property that $\langle  \overline{FM}(b), \lambda \rangle =OV(b)$ and
$
      \langle \overline{FM}(a^k), \lambda\rangle 
      = c_{k}.$
By definition of $U,$ $a^{k} \in U$ for $k = 1, \ldots, n$  and $b \in U.$  However, $\bar{\lambda}$ is an extension of $\lambda$ from  $\overline{FM}(U)$ to $\overline{FM}(\hat{Y}).$  Therefore, for $\psi^{*} = \overline{FM}^{\prime}(\bar{\lambda})$
\begin{eqnarray*}
\langle a^k, \psi^{*} \rangle = \langle  a^k, \overline{FM}'(\bar{\lambda})\rangle = \langle \overline{FM}(a^k), \bar{\lambda}\rangle  =  \langle \overline{FM}(a^k), \lambda \rangle = c_{k}.
\end{eqnarray*}
and similarly
\begin{align*}
\langle b, \psi^{*} \rangle = \langle  b, \overline{FM}'(\bar{\lambda})\rangle = \langle \overline{FM}(b), \bar{\lambda}\rangle  =  \langle \overline{FM}(b), \lambda \rangle =  OV(b)
\end{align*}
and so $\psi^{*}$ is an optimal dual solution to $(\text{DSILP}(\hat{Y}))$ with optimal value $OV(b)$. This is the optimal value of \eqref{eq:SILP}, so there is no duality gap.
\end{proof}

\begin{prop}\label{prop:E-infinity-vector-space}
If $Y$ is a  subspace of $\R^I$ such that $\overline{FM}(\hat{Y}) \subseteq \ell_{\infty}(H)$  then $\hat{Y}$ is a vector space.
\end{prop}

\begin{proof}

If $\hat Y$ is empty we are trivially done. Otherwise let $\bar y$ be any element of $\hat Y$.   Then $-\infty < OV(\bar{y}) < \infty$  so by Proposition~\ref{prop:ov-seq} there exists a sequence $\{h_m\}_{m\in \N}$ in $H$ such that $\tilde a^k(h_m) \to 0$ for $k = 1, \ldots, n$ as $m\to \infty$ which implies $\lim_{m\to\infty} \sum_{k=\ell}^n |\tilde a^k(h_m)| = 0.$ The only purpose of $\bar y$ is to generate the sequence $\left\{h_m\right\}$, which is used below.  

Consider $x,y \in \hat Y$, then $OV(x+y) \leq OV(x) + OV(y) < \infty$ by sublinearity of $OV$. We now show that $-\infty < OV(x+y)$. If $I_3$ is nonempty, then $S(x+y) > -\infty$ and therefore, $OV(x+y) \geq S(x+y) > -\infty$. If $I_3$ is empty, then $OV(x+y) = L(x+y)$ and it suffices to show $L(x+y) >-\infty$. Let $\tilde x = \overline{FM}(x)$ and $\tilde y = \overline{FM}(y)$. By hypothesis, there exists an $N > 0$ such that $ ||\tilde x||_\infty <  N $ and  $ ||\tilde y||_\infty <  N.$
For  any $\delta > 0,$
\begin{eqnarray*}
\omega(\delta,x+y) &=& \sup\{ \tilde{x}(h) + \tilde{y}(h) -  \delta \sum_{k=\ell}^{n} |\tilde{a}^k(h)| \, : \, h \in I_{4} \} \\
&\ge& \sup\{\tilde{x}(h_{m}) +\tilde{y}(h_{m}) - \delta  \sum_{k=\ell}^{n} |\tilde{a}^k(h_{m})| \, : \, m \in \N \} \\
&\ge& \sup\{ -2N - \delta \sum_{k=\ell}^{n} |\tilde{a}^k(h_{m})| \, : \, m \in \N \}  \\
&=& \sup\{  - \delta \sum_{k=\ell}^{n} |\tilde{a}^k(h_{m})| \, : \, m \in \N \} - 2N \\
&=& \delta \sup\{  -  \sum_{k=\ell}^{n} |\tilde{a}^k(h_{m})| \, : \, m \in \N \} -  2N \\
&=& -2N
\end{eqnarray*}
where the last equality comes from the fact that  $ -  \sum_{k=\ell}^{n} |\tilde{a}^k(h_{m})| \le 0$ for all $m \in \N$ and this implies $\sup\{  -  \sum_{k=\ell}^{n} |\tilde{a}^k(h_{m})| \, : \, m \in \N \} \le 0$.  Then $\sum_{k=\ell}^{n} |\tilde{a}^k(h_m)| \to 0$  implies that this supremum is zero. Therefore 
\begin{eqnarray*}
L(x + y) = \lim_{\delta \rightarrow \infty} \omega(\delta) \ge -2N > -\infty.
\end{eqnarray*}

We now confirm that for all $y \in \hat Y$ and $\alpha\in \R$, $-\infty < OV(\alpha y) < \infty$. If $\alpha > 0$ then $OV(\alpha y) = \alpha OV(y)$ by sublinearity of $OV$ and the result follows. Thus, it suffices to check that $-\infty < OV(-y) < \infty$ for all $y \in \hat Y$. By sublinearity of $OV$, $OV(y) + OV(-y) \geq OV(0) = 0$ by Remark~\ref{rem:finite-OV}. Thus, $OV(-y) \geq -OV(y) > -\infty$. 
We now show that $S(-y), L(-y) < \infty$ which  implies  $OV(-y) = \max\{S(-y), L(-y)\} < \infty$. By hypothesis, there exists  $N > 0$ such that $ ||\tilde y||_\infty <  N $. Therefore, $S(-y) = \sup\{-\tilde y(h): h \in I_3\} < N < \infty$. Finally, for every $\delta \geq 0$, 
\begin{eqnarray*}
\omega(\delta,-y) &= &\sup\{-\tilde y(h) - \delta\sum_{k=\ell}^n|\tilde a^k(h)|: h \in I_4\} \\ & \leq &\sup\{-\tilde y(h): h \in I_4\} < N < \infty.
\end{eqnarray*}
This implies $L(-y) = \lim_{\delta \rightarrow \infty} \omega(\delta,-y)  < N <\infty.$
\end{proof}

Theorem~\ref{thm:SD-ell-infty} is an immediate consequence of Theorem~\ref{theorem:extend-SD-Y} and Proposition~\ref{prop:E-infinity-vector-space}.

\begin{theorem}\label{thm:SD-ell-infty}
Suppose the constraint space $Y$ for~\eqref{eq:SILP} is such that $\overline{FM}(\hat{Y})\subseteq \ell_\infty(H)$.   Then  for any $b \in \hat{Y}$ the dual problem $(\text{DSILP}(\hat{Y}))$ is solvable and (SD) holds for the dual pair \eqref{eq:SILP}--$(\text{DSILP}(\hat{Y})).$ \end{theorem}

\begin{remark}\label{remark:sufficient-condition-for-theorem-SD-ell-infinity} 
The hypotheses Proposition~\ref{prop:E-infinity-vector-space} and of Theorem~\ref{thm:SD-ell-infty} look rather technical, we make two remarks about how to verify these conditions.

\begin{enumerate}
\item The hypotheses Proposition~\ref{prop:E-infinity-vector-space} and of Theorem~\ref{thm:SD-ell-infty}   require  $\overline{FM}(\hat{Y})\subseteq \ell_\infty(H).$  However, it may be easier to show    $\overline{FM}(Y)\subseteq \ell_\infty(H)$ which implies  $\overline{FM}(\hat{Y})\subseteq \ell_\infty(H)$ since  $\hat{Y} \subseteq Y.$  For example, if $Y$ is an $\ell_p$ space for some $1 \le p \le \infty$ and  $\overline{FM}(Y)\subseteq \ell_\infty(H)$ then there is a zero duality gap for all $b$ for which~\eqref{eq:SILP} is feasible and bounded.

\item If \eqref{eq:SILP} has $n$ variables then a Fourier-Motzkin multiplier vector has at most $2^n$ nonzero components. Therefore, if the constraint space $Y \subseteq \ell_{\infty}(I) $ and  the nonzero components of the  multiplier vectors $u$ obtained by the Fourier-Motzkin elimination process have a common upper bound $N,$ then  we satisfy the  condition $\overline{FM}(Y) \subseteq \ell_\infty(H) $ in Proposition~\ref{prop:E-infinity-vector-space} and Theorem~\ref{thm:SD-ell-infty}.  Checking that the nonzero components of the  multiplier vectors $u$ obtained by Fourier-Motzkin elimination process have a common upper bound $N$ is verifiable through the Fourier-Motzkin procedure.
\end{enumerate}
\end{remark}

\begin{example}[Example~\ref{ex:duality-gap-cannot-lift}, continued]\label{ex:duality-gap-cannot-lift-2}
Recall that (SD) fails in Example~\ref{ex:duality-gap-cannot-lift}. In this case, $a^{1}, a^{2}, b \in \ell_{\infty}$ (indeed in $\ell_2$)  however the condition $\overline{FM}(\hat{Y}) \subseteq \ell_\infty(H)$ fails since the Fourier-Motzkin  multiplier vectors  are $(1, 0, \ldots, 0, i, 0, \ldots)$ for all $i \in \N$ and $\overline{FM}(-e) \not\in \ell_\infty(H)$  for $e = (1, 1, \ldots,)$   but $-e \in \hat{Y}.$
\end{example}

\subsection{An Example where (SD) holds but (DP) fails}

In Example~\ref{example:karney-modified}  we illustrate a case where  (SD) holds but (DP) fails. In the following subsection we provide sufficient conditions that guarantee when (DP) holds.  

\begin{example}\label{example:karney-modified}
Consider the following modification of Example~1 in Karney \cite{karney81}.
\begin{align}\label{eq:karney-modified}
\begin{array}{rcccl}
\inf x_{1} &&&& \\
x_{1}&& &\ge& -1 \\
&-x_{2} &&\ge& -1 \\
&&-x_{3} &\ge& -1 \\
x_{1} &+ x_{2} & &\ge& 0 \\
x_{1} &- \frac{1}{i} x_{2} &+ \frac{1}{i^{2}} x_{3} &\ge& 0, \quad i = 5, 6, \ldots
\end{array}
\end{align}
In this example $I = \N$. The smallest of the standard constraint spaces that contains the columns and right-hand-side of \eqref{eq:karney-modified} is $\mathfrak c$. To see this note that the first column in the sequence, $(1,0,0,1,1,\dots)$, in not an element of  $\ell_p$  (for $1 \le p < \infty$) and is also not contained in $\mathfrak c_0$. It is easy to check that the columns and the  right hand side  lie in $\mathfrak c$. We show that (SD) holds with $(\text{DSILP}(\mathfrak c))$ but (DP) fails. Then, by Proposition~\ref{prop:mononticity}, (DP) fails for any sequence space that contains $\mathfrak c,$ including $\ell_\infty$.

Our analysis uses the Fourier-Motzkin elimination procedure. First write the constraints of the problem in standard form 
\begin{align*}
\begin{array}{rrcccrl}
z &-x_{1}&&&\ge& 0& b_{0}\\
&x_{1}&& &\ge& -1&  b_{1}\\
&&-x_{2} &&\ge& -1& b_{2} \\
&&&-x_{3} &\ge& -1 &b_{3}\\
&x_{1} &+ x_{2} & &\ge& 0&b_{4} \\
&x_{1} &- \frac{1}{i} x_{2} &+ \frac{1}{i^{2}} x_{3} &\ge& 0 & b_{i}, \quad  i = 5, 6, \ldots,
\end{array}
\end{align*}
and eliminate $x_{3}$ to yield (tracking the multipliers on the constraints to the right of each constraint)
\begin{align*}
\begin{array}{rrccrl}
z &-x_{1}&&\ge& 0& b_{0}\\
&x_{1}& &\ge& -1&  b_{1}\\
&&-x_{2} &\ge& -1& b_{2} \\
&x_{1} &+ x_{2}  &\ge& 0&b_{4} \\
&x_{1} &- \frac{1}{i} x_{2}  &\ge& -\frac{1}{i^{2}} & (\frac{1}{i^{2}})b_{3} + b_{i},   \quad  i = 5, 6, \ldots ,
\end{array}
\end{align*}
then $x_2$ to give
\begin{align*}
\begin{array}{rrcrl}
z &-x_{1}&\ge& 0& b_{0}\\
&x_{1} &\ge& -1&  b_{1}\\
&x_{1}   &\ge& -1&b_{2} + b_{4} \\
&\frac{(1+i)}{i}x_{1}   &\ge& -\frac{1}{i^{2}} & (\frac{1}{i^{2}})b_{3} + (\frac{1}{i})b_{4} + b_{i},   \quad  i = 5, 6, \ldots,
\end{array}
\end{align*}
and finally $x_1$ to give 
\begin{equation}\label{eq:karney-modified-projected}
\begin{array}{rrrl}
 z&\ge& -1&  b_{0} + b_{1}\\
 z  &\ge& -1&b_{0} + b_{2} + b_{4} \\
z &\ge& \frac{-1}{i(1 + i)} &  b_{0} +  \frac{b_{3}}{i(1 + i)} + \frac{b_{4}}{(1+i)} + \frac{i b_{i}}{(1+i)},   \quad  i = 5, 6, \ldots 
\end{array}
\end{equation}

We first claim that (SD) holds.  The components of the Fourier-Motzkin multipliers (which can be read off the right side of~\eqref{eq:karney-modified-projected}) have an upper bound of 1.
 By Remark~\ref{remark:sufficient-condition-for-theorem-SD-ell-infinity}  the hypotheses of Theorem~\ref{thm:SD-ell-infty} hold and  we have (SD). 

We now show that (DP) fails.  We do this by showing that there is a unique optimal dual solution (Claim 1) and that (DP) fails for this unique solution (Claim 2). 

\vskip 7pt
\textbf{Claim 1.} The limit functional $\psi_{0 \oplus 1}$ (using the notation set for dual linear functionals over $\mathfrak c$ introduced in Section~\ref{s:preliminaries}) is the unique dual optimal solution to $(\text{DSILP}(\mathfrak c))$. 
\vskip 7pt

Recall that every positive dual solution in $\mathfrak c$ has the form $\psi_{w \oplus r}$ where $w \in \ell^1_+$ and $r \in \R$ and $\psi_{w \oplus r}(y) = \sum_{i = 1}^\infty w_iy_i + ry_\infty$ for every convergent sequence $y$ with limit $y_\infty$. The constraints to $(\text{DSILP}(\mathfrak c))$ are written as follows
\begin{align*}
\psi_{w\oplus r}(a^1) = 1, \quad \psi_{w\oplus r}(a^2) = 0, \quad \psi_{w\oplus r}(a^3) = 0.
\end{align*}
This implies the following about $w$ and $r$ for dual feasibility
\begin{align*}
w_1 + w_4 + \sum_{i = 5}^\infty w_i + ra^1_\infty &= 1 \\
-w_2 + w_4 - \sum_{i = 5}^\infty \frac{w_i}{i} + ra^2_\infty &= 0 \\
-w_3  - \sum_{i = 5}^\infty \frac{w_i}{i^2} + ra^3_\infty &= 0
\end{align*}
which simplifies to 
\begin{align}
w_4 & = 1 - w_1 - \sum_{i = 5}^\infty w_i - r \label{eq:feasibility-1}\\
w_4 & = w_2 + \sum_{i = 5}^\infty \frac{w_i}{i} \label{eq:feasibility-2} \\
0 & = w_3 + \sum_{i = 5}^\infty \frac{w_i}{i^2} \label{eq:feasibility-3}
\end{align}
by noting $a^1_\infty = 1$ and $a^2_\infty = a^3_\infty = 0$. The dual objective value for a feasible $\psi_{w \oplus r}$ is
\begin{align*}
\psi_{w \oplus r}(b) = - w_1 - w_2 - w_3
\end{align*}
since $b_\infty = 0$. 

Clearly, $\psi_{0 \oplus 1}$ is feasible ($w = 0$ and $r = 1$ trivially satisfies \eqref{eq:feasibility-1}--\eqref{eq:feasibility-3}) with an objective value of $0$. Now consider an arbitrary dual solution $\psi_{w \oplus r}$. If any one of $w_1, w_2, w_3 > 0$ then $\psi_{w \oplus r}(b) < 0$ (recall that $w \ge 0$) and so $\psi_{w \oplus r}$ is not dual optimal since $\psi_{0 \oplus 1}$ yields a greater objective value. This means we can take $w_1 = w_2 = w_3 = 0$ in any optimal dual solution. Combined with   \eqref{eq:feasibility-3} this implies $\sum_{i = 5}^\infty \frac{w_i}{i^2} = 0.$    Since $w_i \ge 0$ this implies $w_i = 0$ for $i = 5, 6, \dots$. From \eqref{eq:feasibility-2} this implies $w_4 = 0$. Thus,  in every dual optimal solution  $w = 0$ and \eqref{eq:feasibility-1} implies $r = 1$. Therefore the limit functional $\psi_{0 \oplus 1}$ is the unique optimal dual solution, establishing the claim. \quad $\dagger$

The limit functional is an optimal dual solution with an objective value of $0$ which is also the optimal primal value since (SD) holds. 
Next we argue that (DP) fails. Since the limit functional is the unique optimal dual solution, it is the only allowable $\psi^*$ in \eqref{eq:dual-pricing}. This observation makes it easy to verify that (DP) fails. We show that \eqref{eq:dual-pricing} fails for $\psi_{0 \oplus 1}$ and $d = (0,0,0,1,0,\dots)$.  This perturbation vector $d$ leaves the problem unchanged except for fourth constraint, which becomes $x_{1} + x_{2} \ge \epsilon$. 
\vskip 7pt
\textbf{Claim 2.}
For all sufficiently small $\epsilon > 0$, the primal problem with the new right-hand-side vector $b + \epsilon d$ for $d = (0,0,0,1,0,\dots)$ is feasible and has a primal objective function value $OV(b + \epsilon d)$ strictly greater than zero. 
\vskip 7pt

Observe from~\eqref{eq:karney-modified-projected}  that $I_{1}$ and $I_{2}$ are empty  and that  the  primal  is feasible  for right-hand-side vector $b + \epsilon d$ for all $\epsilon.$  The third set of inequalities in~\eqref{eq:karney-modified-projected}  are
 \begin{eqnarray*}
 z \ge \frac{-1}{i(1 + i)}   b_{0} +  \frac{b_{3}}{i(1 + i)} + \frac{b_{4}}{(1+i)} + \frac{i b_{i}}{(1+i)},   \quad  i = 5, 6, \ldots 
 \end{eqnarray*} 
  When  $b_{4}$ is changed from 0 to $\epsilon$  we have $b_{0} = 0,$  $b_{1} = b_{2} = b_{3} = -1,$  $b_{4} = \epsilon,$ and $b_{i}= 0.$ These values give
  \begin{eqnarray*}
 z \ge    \frac{-1}{i(1 + i)} + \frac{\epsilon}{(1+i)}     =   \frac{1}{(1+ i)} \left( \epsilon - \frac{1}{i}  \right),   \quad  i = 5, 6, \ldots 
 \end{eqnarray*} 
Let $\epsilon = 1/N$ for a positive integer $N \ge 3.$  Define $\hat i = 2/\epsilon = 2 N$.  Then constraint $\hat i$ is
\begin{eqnarray*}
z \ge \frac{1}{(\frac{2}{\epsilon}+1)} \left(\epsilon - \frac{1}{\frac{2}{\epsilon}}  \right) =   \frac{1}{(\frac{2}{\epsilon}+1)} \left(\frac{\epsilon}{2} \right) > 0. 
\end{eqnarray*}   
This constraint is a lower bound on the objective value of the primal and this implies that $OV(b + \frac{1}{N} d) \ge \frac{1}{(\frac{2}{\epsilon}+1)} \left(\frac{\epsilon}{2} \right) > 0$. This establishes the claim. \quad $\dagger$

To show \eqref{eq:dual-pricing} does not hold,  observe $d$  has finite support so the that limit functional evaluates $d$ to zero.  That is, $\psi_{0 \oplus 1}(d) = 0$. This implies that for all sufficiently small $\epsilon$,
\begin{align*}
OV(b) + \epsilon\psi_{0 \oplus 1}(d) = 0 < OV(b + \epsilon d),
\end{align*}
where the inequality follows by Claim 2. Hence, there does not exist an $\hat \epsilon > 0$ such that \eqref{eq:dual-pricing} holds for $\psi^* = \psi_{0 \oplus 1}$ and $d = (0,0,0,1,0,\dots)$. This implies that (DP) fails. 
\quad $\triangleleft$  
\end{example}

\subsection{Dual pricing in extended constraint spaces}

The fact that (DP) fails for this example is intuitive. The structure of the primal is such that the only dual solution corresponds to the limit functional. However, the value of the limit functional is unchanged by  perturbations to a finite number of constraints. Since the primal optimal value changes under finite support perturbations, this implies that the limit functional cannot correctly ``price'' finite support perturbations. 

 Despite the existence of many sufficient conditions for (SD) in the literature, to our knowledge sufficient conditions to ensure (DP) for semi-infinite programming have only recently been considered for the finite support dual \eqref{eq:FDSILP} (see Goberna and L\'{o}pez \cite{goberna2014post} for a summary of these results). We contrast our results with those in Goberna and L\'{o}pez \cite{goberna2014post} following the proof of Theorem~\ref{theorem:sufficient-conditions-dual-pricing-alt}. Our sufficient conditions for (DP),  based on the output \eqref{eq:J_system} of the Fourier-Motzkin elimination procedure, are

\begin{enumerate}[DP.1]

\item \label{item:dp-condition-sup-S-vertical}   
If $I_3 \neq \emptyset$  and
$\mathcal H_S := \{\{h_{m}\}_{m \in \N} \subseteq I_{3} \text{ and } \limsup \{\tilde{b} (h_{m})\}_{m \in \N} < S(b) \}$ 
then 
\begin{align*}
\sup \{ \limsup \{\tilde{b} (h_{m})\}_{m \in \N}  : \{h_{m}\}_{m \in \N} \in \mathcal H_S \}   < S(b).
\end{align*}
\item \label{item:dp-condition-sup-L-vertical}   
If $I_4 \neq \emptyset$ and 
\begin{align*}
\mathcal H_L :=\{ \{h_m\}_{m \in \N} \subseteq I_4 :  \limsup \{\tilde{b} (h_{m})\}_{m \in \N} < L(b) \text{ and } \lim_{m\to \infty} \sum_{k=\ell}^n |\tilde a^k(h_m)| = 0\}
\end{align*}
then 
\begin{align*}
\sup \{ \limsup \{\tilde{b} (h_{m})\}_{m \in \N} : \{h_{m}\}_{m \in \N} \in \mathcal H_L \} < L(b).
\end{align*}

\end{enumerate}

By Lemmas~\ref{lem:seq-L}  and~\ref{lem:seq-S}, subsequences $\{\tilde{b}(h)\}$ with  the indices $h$ in $I_{3}$ or $I_{4}$ are bounded above by $S(b)$ and $L(b),$ respectively,  and in the case of $L(b)$, $\tilde a^k(h) \to 0$ for all $k = 1, \dots, n$. Conditions  DP.\ref{item:dp-condition-sup-S-vertical}-DP.\ref{item:dp-condition-sup-L-vertical} require that limit values of these subsequences that do not achieve $S(b)$ or $L(b)$ (depending on whether the sequence is in $I_3$ or $I_4$, respectively) do not become arbitrarily close to $S(b)$ or $L(b)$. 

\begin{remark}\label{remark:thoughts-on-DP-1-DP-2}
In the case of  Condition DP.\ref{item:dp-condition-sup-S-vertical}, given $h \in I_3$    we may take $h_m = h$ for all $m \in \N$ and then $\limsup \{\tilde{b} (h_{m})\}_{m \in \N} = \tilde{b} (h).$  Then   Condition DP.\ref{item:dp-condition-sup-S-vertical}  becomes $\sup \{ \tilde{b}(h) \, : \,  h \in I_{3}  \text{ and } \tilde{b}(h) < S(b) \} < S(b)$ when $I_3 \neq \emptyset$.  This condition can only hold if the supremum of the $\tilde{b}(h)$ is achieved over $I_{3}.$  A similar conclusion does not hold for DP.\ref{item:dp-condition-sup-L-vertical}. In this case $\{ h_{m} \}_{m \in \N}$ cannot be a sequence of identical indices if  $\lim_{m\to \infty} \sum_{k=\ell}^n |a^k(h_m)| = 0$  since  $\sum_{k=\ell}^{n} |\tilde{a}^k(h_{m})| \neq 0$ for all $h_{m} \in I_{4}.$
\end{remark}

The proof of theorem uses three technical lemmas (Lemmas~\ref{lem:conv-comb}--\ref{lemma:s-vertical-pricing}) found in the appendix.

\begin{theorem}\label{theorem:sufficient-conditions-dual-pricing-alt}
Consider an instance of \eqref{eq:SILP} that is feasible and bounded for right-hand-side $b.$
Suppose the constraint space $Y$ for~\eqref{eq:SILP} is such that $\overline{FM}(\hat{Y})\subseteq \ell_\infty(H)$    and Conditions DP.\ref{item:dp-condition-sup-S-vertical} and DP.\ref{item:dp-condition-sup-L-vertical} hold. Then   property (DP) holds for  \eqref{eq:SILP}.
\end{theorem}

\begin{proof}

Assume $d \in \hat{Y}$ is a perturbation vector such that $b + d$ is feasible. We show  there exists an optimal dual solution $\psi^{*}$ to \eqref{eq:DSILPprime} and an $\hat \epsilon > 0$ such that
\begin{align*}
OV(b + \epsilon d) = \psi^{*}(b + \epsilon d) = OV(b) + \epsilon \psi^{*}(d)
\end{align*}
for all $\epsilon \in [0, \hat \epsilon].$ There are several cases to consider. 

\noindent \underline{\em Case 1: $L(b)  >  S(b)$.} By hypothesis $\overline{FM}(d) = \tilde{d} \in \ell_{\infty}(H)$ and this implies $\sup_{h \in I_{3}} | \tilde{d}(h) | < \infty.$ Thus, $S(d) < \infty$. Then $L(b)  >  S(b)$ implies  there exists an $\epsilon_{1} > 0$ such that $L(b) > S(b) + \epsilon S(d)$ for all $\epsilon \in [0, \epsilon_{1}]$. However, by Lemma~\ref{lem:convex-L}, $S(y)$ is a sublinear function of $y$ so $S(b) + \epsilon S(d) \ge S(b + \epsilon d).$ Define $\beta := \min_{\epsilon\in[0,\epsilon_1]}L(b) - S(b+ \epsilon d)\geq \min_{\epsilon\in[0,\epsilon_1]} L(b) - S(b) - \epsilon S(d)$. Since the function $L(b) - S(b) - \epsilon S(d)$ is linear and it is strictly positive at the end points of $[0,\epsilon_1]$, this implies $\beta > 0$.

Again, $\tilde{d} \in \ell_{\infty}(H)$ implies the existence of $\epsilon_{2} > 0$ such that $ \epsilon_{2} \sup_{h \in I_{4}} | \tilde{d}(h) | < \beta/2.$   Let $\epsilon_{3} = \min\{\epsilon_{1}, \epsilon_{2} \}.$  Then for all $\epsilon \in [0, \epsilon_{3}]$
\begin{eqnarray*}
\begin{array}{rcl} 
L(b + \epsilon d) & = &\lim_{\delta\to\infty}\sup \{\tilde{b}(h) + \epsilon \tilde d(h)- \delta \sum_{k=\ell}^{n} |\tilde{a}^k(h)| \, : \, h \in I_4 \} \\ &\geq &\lim_{\delta\to\infty}\sup \{\tilde{b}(h) - \frac\beta2 - \delta \sum_{k=\ell}^{n} |\tilde{a}^k(h)| \, : \, h \in I_4 \} \\ & = & \lim_{\delta\to\infty}\sup \{\tilde{b}(h)  - \delta \sum_{k=\ell}^{n} |\tilde{a}^k(h)| \, : \, h \in I_4 \} - \frac\beta2 \\ & = & L(b) -\frac\beta2 \\
&>&  S(b + \epsilon d).
\end{array}
\end{eqnarray*}
A similar argument gives $L(b + \epsilon d) < L(b) + \frac\beta2$ so  $L(b + \epsilon d) < \infty.$

 By hypothesis \eqref{eq:SILP} is  feasible  so by Theorem~\ref{theorem:fm-primal-value}, $OV(b) = \max\{S(b), L(b)\}.$  Then  $L(b)  >  S(b)$ implies $L(b) > -\infty.$  Thus $-\infty <  L(b), L(b + \epsilon_{3} d) < \infty.$  Thus, the hypotheses of Lemma~\ref{lemma:L-vertical-pricing} hold.   Now apply Lemma~\ref{lemma:L-vertical-pricing} and observe there is a $\hat{\epsilon}$ which we can take to be less than $\epsilon_{3}$ and a sequence $\{h_m\}_{m\in \N} \subseteq I_4$ such that for all $\epsilon \in [0,\hat \epsilon]$ 
\begin{align*}
\tilde b(h_m) \to L(b),  \tilde d_{\epsilon} (h_m) \to L(b + \epsilon d),  \text{ and } \sum_{k=\ell}^{n} |\tilde{a}^k(h_m)| \to 0
\end{align*}
where $\tilde d_{\epsilon} = \overline{FM}(b + \epsilon d)$.

We have also shown  for all $\epsilon \in [0, \epsilon_{3}],$ $L(b + \epsilon d) > S(b + \epsilon d).$ Then by Theorem~\ref{theorem:fm-primal-value} $OV(b + \epsilon d) = L(b + \epsilon d).$  Using the sequence $\{h_m\}_{m\in \N} \subseteq I_4$  define the linear functional $\lambda$ as in~\eqref{eq:projected-system-linear-functional}.  Then extend this linear functional as in Theorem~\ref{theorem:extend-SD-Y} and use the adjoint of the $\overline{FM}$ operator to get the linear functional $\psi^{*}$ with the property that  $OV(b + \epsilon d) = \psi^{*}(b + \epsilon d)$ for all $\epsilon \in [0, \hat{\epsilon}].$

\vskip 5pt
\noindent \underline{\em Case 2: $S(b)  > L(b)$.} This case follows the same proof technique as in the $L(b) > S(b)$ case but invoke  
Lemma~\ref{lemma:s-vertical-pricing}
instead of Lemma~\ref{lemma:L-vertical-pricing}.

\vskip 5pt
\noindent \underline{\em Case 3: $S(b) = L(b)$.} By Lemma~\ref{lemma:L-vertical-pricing}   there exists $\hat \epsilon_{L} > 0$ and a sequence $\{h_m\}_{m\in \N} \subseteq I_4$ such that for all $\epsilon \in [0,\hat \epsilon_{L}]$   
\begin{align*}
\tilde d_{\epsilon} (h_m) \to L(b + \epsilon d),  \text{ and } \sum_{k=\ell}^{n} |\tilde{a}^k(h_m)| \to 0
\end{align*}
where $\tilde d_{\epsilon} = \overline{FM}(b + \epsilon d)$.  

Likewise by Lemma~\ref{lemma:s-vertical-pricing} 
there exists $\hat \epsilon_{S} > 0$ and a sequence $\{g_m\}_{m\in \N} \subseteq I_3$ such that for all $\epsilon \in [0,\hat \epsilon_{S}]$
\begin{align*}
  \tilde d_{\epsilon} (g_m) \to S(b + \epsilon d)\end{align*}
where $\tilde d_{\epsilon} = \overline{FM}(b + \epsilon d)$.

Now let $\hat \epsilon = \min \{ \hat \epsilon_{L},  \hat \epsilon_{S} \}.$ By Lemma~\ref{lem:conv-comb}, for all $\epsilon \in (0, \hat \epsilon]$, $S(b+\epsilon d)$ and $L(b + \epsilon d)$ are the same convex combinations of $S(b), S(b + \hat \epsilon d)$ and $L(b), L(b + \hat \epsilon d)$ respectively. There are now three possibilities. 
First, if $S(b + \hat \epsilon d) = L(b +  \hat \epsilon d)$ then $S(b +  \epsilon d) = L(b +   \epsilon d)$ for all $\epsilon \in (0, \hat \epsilon]$ and we have alternative optimal dual linear functionals generated from the $\{g_{m}\}$ and $\{h_{m}\}$ sequences.   
Second, if $S(b + \hat \epsilon d) >  L(b + \hat \epsilon d)$ then $S(b +  \epsilon d) > L(b +   \epsilon d)$ for all $\epsilon \in (0, \hat \epsilon]$ and  the  dual linear functional generated from the $\{g_{m}\}$   sequence will satisfy the dual pricing property.
Third, if $S(b + \hat \epsilon d) <  L(b + \hat  \epsilon d)$ then $S(b +  \epsilon d) < L(b +   \epsilon d)$ for all $\epsilon \in (0, \hat \epsilon]$ and    the  dual linear functional generated from the $\{h_{m}\}$   sequence will satisfy the dual pricing property. 
\end{proof}

The following two examples illustrate that neither of DP.\ref{item:dp-condition-sup-S-vertical} nor DP.\ref{item:dp-condition-sup-L-vertical} are redundant conditions.

\begin{example}[Example~\ref{example:karney-modified}]\label{ex:karney-modified-continued}
 Example~\ref{example:karney-modified} did not have the (DP) property.  Recall for this example that $OV(b) = S(b) = 0.$  Consider the projected system \eqref{eq:karney-modified-projected}.    Condition DP.\ref{item:dp-condition-sup-L-vertical} is satisfied vacuously since $I_4 = \emptyset$. However, Condition DP.\ref{item:dp-condition-sup-S-vertical}  does not hold because  $-1/i(1+i) < 0 = S(b),$ for $i = 5, 6, \dots,$  but the supremum over all $i$ is zero.  That is,  $\sup \{ \tilde{b}(h) \, : \,  h \in I_{3}  \text{ and } \tilde{b}(h) < 0 \} = 0 = S(b).$  See the comments in Remark~\ref{remark:thoughts-on-DP-1-DP-2}. \quad $\triangleleft$ 
\end{example}

\begin{example}\label{example:sufficient-conditions-vertical-dual-pricing} 
Consider the following \eqref{eq:SILP}
\begin{eqnarray*}
\inf x_{1} && \\
 x_{1} + \frac{1}{m + n}x_{2} &\ge& -\frac{1}{n^{2}}, \quad (m, n) \in I
\end{eqnarray*}
whose constraints are indexed by $I = \{ (m,n) \, : \, (m,n) \in \N \times \N \}.$ Putting into standard form gives
\begin{eqnarray*}
\inf z && \\
  z - x_{1} &\ge& 0\\
x_{1} + \frac{1}{m + n}x_{2} &\ge& -\frac{1}{n^{2}}, \quad (m, n) \in I.
\end{eqnarray*}
Apply Fourier-Motzkin elimination, observe $H = I_{4} = I,$ and obtain
\begin{eqnarray*}
\inf z && \\
z+ \frac{1}{m + n}x_{2} &\ge& -\frac{1}{n^{2}}, \quad (m, n) \in I_{4}.
\end{eqnarray*}
In this case $I_{3} = \emptyset$ so DP.\ref{item:dp-condition-sup-S-vertical} holds vacuously.  We show that DP.\ref{item:dp-condition-sup-L-vertical} fails to hold for this example and that property (DP) does not hold.

In our notation, for an arbitrary but fixed  $\bar n \in \N,$ there are subsequences 
\begin{eqnarray*}
\{\tilde{b}(m, \bar n)\}_{m \in \N} = \{-\frac{1}{\bar{n}^{2}} \}_{m \in \N}\to -\frac{1}{\bar{n}^{2}}, \qquad \{\tilde{a}(m,\bar n)\}_{m \in \N} = \{\frac{1}{m + \bar n}\}_{m \in \N} \to 0.
\end{eqnarray*}

Likewise, for an arbitrary but fixed  $\bar m \in \N,$ there are subsequences 
\begin{eqnarray*}
\{\tilde{b}(\bar m, n)\}_{n \in \N} = \{-\frac{1}{n^{2}}\}_{n \in \N} \to 0, \qquad \{\tilde{a}(\bar m,n)\}_{n \in \N} = \{\frac{1}{\bar m + n}\}_{n \in \N} \to 0.
\end{eqnarray*}

{\bf Claim 1:} An optimal primal solution is $x_{1} =  x_{2} = 0$ with optimal value $z = 0.$  Clearly  $x_{1} = x_{2} = 0$ is a primal feasible solution with objective function value 0 since the right-hand-side vector is negative. Now we argue that the optimal objective value cannot be negative.   In this example only $I_{4}$ is nonempty so $S(b) = -\infty$  and it suffices to show $L(b) = 0.$ In this example, 
\begin{eqnarray*}
\omega(\delta, b) = \sup \left\{ -\frac{1}{n^{2}}  -\frac{\delta}{m + n} \, : \, (m, n) \in I_{4} = \N \times \N  \right\} \le 0.
\end{eqnarray*}
For any subsequence of $\{(m, n)\} \in I_{4},$  $-\frac{\delta}{m + n} \to 0.$  Since $\{-\frac{1}{n^{2}}\}_{n \in \N} \to 0$ for each $m,$ by Lemma~\ref{lem:seq-L} we have $L(b) = 0.$  $\dagger$

\vskip 7pt
We consider perturbation vector $d(m,n) = \tilde{d}(m,n)  = \frac{1}{n}$ for all $(m, n) \in I_{4}.$
\vskip 7pt

{\bf Claim 2:}  For  all $n \in \N$, $L(b +  \frac{2}{n} d) =  \frac{(2/n)^{2}}{4} = \frac{1}{n^2}.$
For a fixed $\hat n \in \N,$ consider the subsequence $\{ m, \hat{n} \}_{m \in \N}$ of  $I_4$ where 
\begin{eqnarray*}
\{\tilde{b}(m, \hat n) + \frac{2}{\hat n} \tilde{d}(m, \hat n) \}_{m \in \N} = \{ -\frac{1}{\hat{n}^{2}} + \frac{2}{\hat n} \frac{1}{\hat n}  \}_{m \in \N} = \{ \frac{1}{\hat n^2} \}_{m \in \N}.
\end{eqnarray*}
Then since $\{(m, \hat n)\} \in I_{4}$ for all $m \in \N$,  $\frac{1}{m + \hat n} \to 0$ as $m \to \infty$, by Lemma~\ref{lem:seq-L}, $L(b +  \frac{2}{\hat n}  d) \ge  \frac{1}{\hat n^2}.$  Now show this is an equality by showing it is the best possible limit value of any sequence.

The maximum value of $\{\tilde{b}(m, n) + \frac{2}{\hat n} \tilde{d}(m, n) \}_{(m,n) \in \N \times \N}$ is given by 
\begin{eqnarray*}
\max_{n} \left(-\frac{1}{n^{2}} + \frac{2}{\hat n n}  \right),
\end{eqnarray*}
which, using simple Calculus, is achieved for $n = \hat n$. This shows that $\tilde{b}(m, n) + \frac{2}{\hat n} \tilde{d}(m, n) \le \frac{1}{\hat n^2}$ for all $(m,n) \in \N \times \N$. From Lemma~\ref{lem:seq-L}, $L(b + \frac{2}{\hat n} d)$ is the limit of some subsequence of elements in $\{\tilde{b}(m, n) + \frac{2}{\hat n} \tilde{d}(m, n) \}_{(m,n) \in \N \times \N}$. Since each element is less than $\frac{1}{\hat n^2}$, $L(b + \frac{2}{\hat n} d) \le \frac{1}{\hat n^2}$. This implies that $L(b +  \frac{2}{\hat n}  d) =  \frac{1}{\hat n^2}.$

\vskip 7pt

{\bf Claim 3:} For this perturbation vector $d,$  there is no dual solution $\psi$ and an $\hat \epsilon > 0$ such that 
\begin{eqnarray*}
OV(b + \epsilon d) = L( b + \epsilon d) = \psi(b + \epsilon d)
\end{eqnarray*}
for all $\epsilon \in [0, \hat \epsilon].$  Assume such a $\psi$ and $\hat \epsilon > 0 $ exists. Consider any $\hat n$ such that $\frac{2}{\hat n} < \hat \epsilon$. By Claim 2, $L(b + \frac{2}{\hat n} d) = \frac{1}{\hat n^2}$, but by the linearity of $\psi$, $\psi(b + \frac{2}{\hat n} d) = \psi(b) + \frac{2}{\hat n} \psi(d)$. Then  $L(b + \frac{2}{\hat n} d) = \psi(b + \frac{2}{\hat n} d)$ implies $\frac{1}{\hat n^2} = \psi(b) + \frac{2}{\hat n} \psi(d)$ for all $\hat n$ such that $\frac{2}{\hat n} < \hat \epsilon$. By Claim 1, $L(b) = 0$ so $\psi(b) =  0.$ Then $\frac{1}{\hat n} = 2\psi(d)$ for all $\hat n$ such that $\frac{2}{\hat n} < \hat \epsilon$. However  $\psi(d)$ is a fixed number and cannot vary with $\hat{n}.$ This is a contradiction and (DP) fails.   \quad $\triangleleft$ 
\end{example}

In \cite{goberna2007sensitivity}, Goberna et al. give sufficient conditions for a dual pricing property.   They use the notation
\begin{eqnarray*}
T(x) := \{  i \in I \, : \,  \sum_{k=1}^{n} a^{k}(i) x = b(i) \} \\
A(x) := \cone\{ (a^{1}(i), \ldots, a^{k}(i)) \, : \, i \in T(x)  \}.
\end{eqnarray*}
Their main results for right-hand-side sensitivity analysis appear as Theorem 4 in~\cite{goberna2007sensitivity} and again as Theorem 4.2.1 in~\cite{goberna2014post}.   In this theorem a key hypothesis (hypothesis (i.a) in the statement of Theorem 4 in~\cite{goberna2007sensitivity}) is that $c \in A(x^{*})$ where $x^{*}$ is a feasible solution to~\eqref{eq:SILP}. We show in Theorem~\ref{theorem:goberna-sensitivity} below that in our terminology (i.a) implies $S(b) \ge L(b)$ and both primal and dual solvability.

\begin{theorem}\label{theorem:goberna-sensitivity}
If~\eqref{eq:SILP} has a feasible solution $x^{*}$ and $c \in A(x^{*})$ then:  (i)  $S(b) \ge L(b),$ (ii)  $S(b) = \sup_{h \in I_{3}}\{ \tilde{b}(h) \}$ is realized, and (iii) $x^{*}$ is an optimal primal solution.
\end{theorem}

\begin{proof}
If $c \in A(x^{*})$ then there exists $\bar{v} \ge 0$ with finite support contained in $T(x^{*})$  such that $\sum_{i \in I} \bar{v}(i) a^{k}(i) = c_{k}$ for $k = 1, \ldots, n.$  By hypothesis, $x^{*}$ is a feasible solution to~\eqref{eq:SILP} and it follows from Theorem 6 in Basu et al.~\cite{basu2013projection} that   $\tilde{b}(h) \le 0$ for all $h \in I_{1}.$  Then by Lemma 5 in the same paper there exists $\bar{h} \in I_{3}$ such that $\tilde{b}(\bar{h}) \ge \sum_{i \in I} \bar{v}(i) b(i).$  More importantly,  the support of $\bar{h}$ is a subset of the support of $\bar{v}.$  Then the support of $\bar{h}$ is contained in $T(x^{*})$ since $\bar{v}_{i} > 0$ implies $i \in T(x^{*}).$  Then for this $\bar{h},$  $v^{\bar{h}}(i) > 0$ for only those $i \in I$ for which constraint $i$ is tight.  Then we aggregate the tight constraints in~\eqref{eq:initial-system-obj-con}-\eqref{eq:initial-system-con} associated with the support of $\bar{h}$ and observe
\begin{eqnarray}\label{eq:goberna-sensitivity}
z = \sum_{k=1}^{n} c_{k} x_{k}^{*} = \sum_{i \in I} v^{\bar{h}}(i) b(i) = \tilde{b}(\bar{h}).  
\end{eqnarray}
It follows from~\eqref{eq:goberna-sensitivity} that $x^{*}$ is an optimal primal solution and $v^{\bar{h}}$ is an optimal dual solution and (i)-(iii) follow.
\end{proof}

The following example satisfies (DP) but (iii) of Theorem~\ref{theorem:goberna-sensitivity} fails to hold since the primal is not solvable.

\begin{example}[Example 3.5 in \cite{basu2014sufficiency}]\label{ex:drop-primal-solvability}
Consider the  \eqref{eq:SILP}
\begin{align}\label{eq:not-primal-optimal}
\begin{array}{rcl}
\inf x_{1} \phantom{+ \tfrac{1}{i^{2}}x_{2} \ } && \\
\phantom{\inf } x_{1} + \tfrac{1}{i^{2}}x_{2} &\ge& \tfrac{2}{i}, \quad i \in \N. 
\end{array}
\end{align}
with constraint space taken to be $\ell_\infty$. We apply the Fourier-Motzkin elimination procedure by putting \eqref{eq:not-primal-optimal} into standard form to yield
\begin{align*}
\begin{array}{rcl}
z - x_{1} \phantom{+ \tfrac{1}{i^{2}}x_{2} \ } &\ge & 0 \\
\phantom{z - } x_{1} + \tfrac{1}{i^{2}}x_{2} &\ge& \tfrac{2}{i}, \quad i \in \N .
\end{array}
\end{align*}
Eliminating $x_1$ gives the projected system:
\begin{align*}
z + \tfrac{1}{i^{2}}x_{2} \ge \tfrac{2}{i}, \quad i \in \N.
\end{align*}
Observe that $H = \N = I_4$ and $x_2$ cannot be eliminated. Since $I_3 = \emptyset$, $S(b) = -\infty$ and so by Theorem~\ref{theorem:fm-primal-value} the optimal value of \eqref{eq:not-primal-optimal} is $L(b)$. Recall that $L(b) = \lim_{\delta \to \infty} \omega(\delta, b)$ where $\omega(\delta, b) = \sup_{i \in \N} \left\{ \tfrac{2}{i} - \tfrac{1}{i^{2}} \delta \right\} \leq \frac{1}{\delta}$, where the inequality was shown in~\cite{basu2014sufficiency}. Also, for a fixed $\delta \geq 0$, $\sup_{i \in \N} \left\{ \tfrac{2}{i} - \tfrac{1}{i^{2}} \delta \right\} \geq 0$ and so $\omega(\delta, b) \geq 0$ for all $\delta\geq 0$. Hence, $0\leq L(b) = \lim_{\delta\to\infty} \omega(\delta, b) \leq \lim_{\delta \to \infty} \tfrac{1}{\delta} = 0$. This implies $L(b) = 0$. However,  the optimal objective value is never attained, since there is no feasible solution with $x_1 = 0.$ 

Next we show that (DP) holds. DP.\ref{item:dp-condition-sup-S-vertical} holds vacuously since $I_3 = \emptyset$.  Also,  DP.\ref{item:dp-condition-sup-L-vertical} holds vacuously because $L(b) = 0$ and $\tilde{b(h)} > 0$ for all $h \in I_{4}.$ 

Observe also that the $FM$ linear operator maps $\ell_{\infty}(\{ 0 \} \cup \N )$ into $\ell_{\infty}(\N)$. To see that this is the case observe that all of the multiplier vectors have exactly two nonzero components and both components are $+1.$ Thus, applying the $FM$ operator to any vector in $\ell_{\infty}(\{ 0 \} \cup \N)$ produces another vector in $\ell_{\infty}(\N)$ since adding any two bounded components produces bounded components. Hence we can apply Theorem~\ref{theorem:sufficient-conditions-dual-pricing-alt} to conclude \eqref{eq:not-primal-optimal} satisfies (DP). \quad $\triangleleft$ 
\end{example}

\section{Conclusion}

This paper explores important duality properties of semi-infinite linear programs over a spectrum of constraint and dual spaces. Our flexibility to different choices of constraint spaces provides insight into how properties of a problem can change when considering difference spaces for perturbations. In particular, we show that \emph{every} SILP satisfies (SD) and (DP) in a very restricted constraint space $U$ and provide sufficient conditions for when (SD) and (DP) hold in larger spaces. 

The ability to perform senstivity analysis is critical for any practical implementation of a semi-infinite linear program because of the uncertainty in data in real life problems. However, there is another common use of (DP). In finite linear programming optimal dual solutions correspond to ``shadow prices'' with economic meaning regarding the marginal value of each individual resource. These marginal values can help govern investment and planning decisions. 

The use of dual solutions as shadow prices  poses difficulties in the case of semi-infinite programming. Indeed, it is not difficult to show Example~\ref{ex:drop-primal-solvability} has a unique optimal dual solution over the constraint space $\mathfrak c$ -- namely, the limit functional $\psi_{0 \oplus 1}$ (the argument for why this is the case is similar to that of Example~\ref{example:karney-modified} and thus omitted). Since (DP) holds in Example~\ref{ex:drop-primal-solvability} this means there is a optimal dual solution that satisfies \eqref{eq:dual-pricing} for every feasible perturbation. This is a desirable result. However, interpreting the limit functional as assigning a ``shadow price" in the standard way is problematic. Under the limit functional the marginal value for each individual resource (and indeed any finite bundle of resources) is zero, but infinite bundles of resources may have positive marginal value.  This makes it difficult to interpret this dual solution as assigning economically meaningful shadow prices to individual constraints. 

In a future work we aim to uncover the mechanism by which such undesirable dual solutions arise and explore ways to avoid such complications. This direction draws inspiration from earlier work by Ponstein \cite{ponstein1981use} on countably infinite linear programs.

\appendix


\section{Appendix}\label{s:appendix}

This appendix contains three technical lemmas used in the proof of Theorem~\ref{theorem:sufficient-conditions-dual-pricing-alt}.

\begin{lemma}\label{lem:conv-comb}
Let $b^1, b^2 \in \R^I$ and $\tilde b^1 = \overline{FM}(b^1)$ and $\tilde b^2 = \overline{FM}(b^2)$. Suppose $\{h_m\}_{m\in \N}$ is a sequence in $I_4$ such that $\lim_{m\to\infty} \tilde b^j(h_m) = L(b^j)$ for $j=1,2$ and $\lim_{m\to\infty} \sum_{k=\ell}^n |\tilde a^k(h_m)| \to 0$. Then for every $\lambda \in [0,1]$, $b_\lambda := \lambda b^1 + (1-\lambda)b^2$ has the property that $$\lim_{m\to\infty} \tilde b_\lambda(h_m) = L(b_\lambda) = \lambda L(b_1) + (1-\lambda)L(b_2),$$ where $ \tilde b_\lambda = \overline{FM}(b_\lambda)$. 

Moreover, suppose $\{h_m\}_{m\in \N}$ is a sequence in $I_3$ such that $\lim_{m\to\infty} \tilde b^j(h_m) = S(b^j)$ for $j=1,2$. Then for every $\lambda \in [0,1]$, $b_\lambda := \lambda b^1 + (1-\lambda)b^2$ has the property that $$\lim_{m\to\infty} \tilde b_\lambda(h_m) = S(b_\lambda) = \lambda S(b_1) + (1-\lambda)S(b_2),$$ where $ \tilde b_\lambda = \overline{FM}(b_\lambda)$. 
\end{lemma}  
\begin{proof}
By Lemma~\ref{lem:convex-L} $L$ is sublinear and therefore convex which implies
\begin{align*}
\begin{array}{rcl}
L(b_\lambda) &\leq &\lambda L(b^1) + (1-\lambda)L(b^2) \\
& = & \lambda\lim_{m\to\infty} \tilde b^1(h_m) + (1-\lambda)\lim_{m\to\infty} \tilde b^2(h_m) \\
& = & \lim_{m\to\infty} (\lambda\tilde b^1(h_m) + (1-\lambda)\tilde b^2(h_m))\\
& \leq & L(\lambda b^1 + (1-\lambda)b^2)\\
& = & L(b_\lambda)
\end{array}
\end{align*}
where the second inequality follows from Lemma~\ref{lem:seq-L}.

Thus, all the inequalities in the above are actually equalities. In particular, $\lim_{m\to\infty} (\lambda\tilde b^1(h_m) + (1-\lambda)b^2(h_m)) = L(b_\lambda) = \lambda L(b_1) + (1-\lambda)L(b_2)$.  Since $\overline{FM}$ is a linear operator,  $\overline{FM}(b_\lambda) = \lambda \overline{FM}(b^1) + (1-\lambda) \overline{FM}(b^2)$ and so $\tilde b_\lambda(h_m) = \lambda\tilde b^1(h_m) + (1-\lambda)b^2(h_m)$ for all $m \in \N$. Hence, 
$\lim_{m\to\infty} \tilde b_\lambda(h_m) = \lim_{m\to\infty} (\lambda\tilde b^1(h_m) + (1-\lambda)b^2(h_m)) = L(b_\lambda)$.

For the second part of the result concerning $S$, completely analogous reasoning (except now $\left\{h_m\right\}$ is a sequence in $I_3$ instead of $I_4$ and we use Lemma~\ref{lem:seq-S} instead of Lemma~\ref{lem:seq-L}) shows $\lim_{m\to\infty} \tilde b_\lambda(h_m) = S(b_\lambda)$.
\end{proof}

\begin{lemma}\label{lemma:L-vertical-pricing}
Let $b, d \in \ell_\infty(I)$ such that $-\infty < L(b), L(b+d) < \infty.$ Assume   DP.\ref{item:dp-condition-sup-L-vertical}  and that   $\overline{FM}(\ell_\infty(I)) \subseteq \ell_\infty(H)$. Then there exists $\hat \epsilon > 0$ and a sequence $\{h_m\}_{m\in \N} \subseteq I_4$ such that for all $\epsilon \in [0,\hat \epsilon]$:
\begin{align*}
 \tilde d_{\epsilon} (h_m) \to L(b + \epsilon d)  \text{ and } \sum_{k=\ell}^{n} |\tilde{a}^k(h_m)| \to 0
\end{align*}
where $\tilde d_{\epsilon} := \overline{FM}(b + \epsilon d)$.  
\end{lemma}

\begin{proof}
Define 
 \begin{align*}
\alpha :=  L(b) - \sup \{ \limsup \{\tilde{b} (h_{m})\}_{m \in \N} : \{h_{m}\}_{m \in \N} \in \mathcal H_L \} 
\end{align*}
By hypothesis, $-\infty < L(b) < \infty$   so $I_{4}$ is not empty and then by  assumption DP.\ref{item:dp-condition-sup-L-vertical} $\alpha $  is a positive real number. 
 
\noindent 1. Since  $\tilde d = \overline{FM}(d) \in \ell_\infty(H)$  there exists  $\hat\epsilon > 0$ such that $$\hat\epsilon \sup_{h \in I_4} | \tilde d(h)| < \frac\alpha3.$$
 
\noindent 2. Claim: $L(b) - \frac\alpha3 \leq L(b + \hat\epsilon d) \leq L(b) + \frac\alpha 3$. Proof:
$$\begin{array}{rcl} L(b + \hat\epsilon d) & = &\lim_{\delta\to\infty}\sup \{\tilde{b}(h) + \hat\epsilon\tilde d(h)- \delta \sum_{k=\ell}^{n} |\tilde{a}^k(h)| \, : \, h \in I_4 \} \\ &\geq &\lim_{\delta\to\infty}\sup \{\tilde{b}(h) - \frac\alpha3 - \delta \sum_{k=\ell}^{n} |\tilde{a}^k(h)| \, : \, h \in I_4 \} \\ & = & \lim_{\delta\to\infty}\sup \{\tilde{b}(h)  - \delta \sum_{k=\ell}^{n} |\tilde{a}^k(h)| \, : \, h \in I_4 \} - \frac\alpha3 \\ & = & L(b) -\frac\alpha3 \end{array}$$
Similarly, one can show $L(b + \hat\epsilon d) \leq L(b) + \frac\alpha 3$.

\noindent 3. Consider $\overline{FM}(b + \hat \epsilon d) = \overline{FM}(b) + \hat \epsilon \overline{FM}(d) = \tilde b + \hat\epsilon\tilde d.$  By Claim 2, $L(b + \hat\epsilon d)$ is finite. By Lemma~\ref{lem:seq-L}, there exists a sequence $\{h'_m\}$ such that $\tilde b(h'_m) + \hat\epsilon\tilde d(h'_m) \to L(b + \hat\epsilon d)$ and $\sum_{k=\ell}^{n} |\tilde{a}^k(h'_m)| \to 0$. 

\noindent 4. Claim:  $\limsup \{ \tilde b(h'_m) \}_{m \in \N} = L(b).$ Proof: first show  $\limsup \{ \tilde b(h'_m) \}_{m \in \N} \le L(b).$  If $$\limsup \{ \tilde b(h'_m) \}_{m \in \N} > L(b)$$ then there is subsequence of indices $\{ h^{\prime \prime}  \}_{m \in \N}$ from $\{ h^{\prime }  \}_{m \in \N}$ such that 
$\lim_{m \rightarrow \infty }\tilde b(h''_m) > L(b).$  But  $\sum_{k=\ell}^{n} |\tilde{a}^k(h'_m)| \to 0$ so  $\sum_{k=\ell}^{n} |\tilde{a}^k(h''_m)| \to 0.$  This directly contradicts Lemma~\ref{lem:seq-L} so we conclude   $\limsup \{ \tilde b(h'_m) \}_{m \in \N} \le L(b).$  

Since $\limsup \{ \tilde b(h'_m) \}_{m \in \N} \le L(b)$  it suffices to show  $\limsup \{ \tilde b(h'_m) \}_{m \in \N} = L(b)$    by showing $\limsup \{ \tilde b(h'_m) \}_{m \in \N}$ cannot be strictly less than $L(b).$   From Step 3. above,  we know  $\{ \tilde b(h'_m) + \hat\epsilon\tilde d(h'_m) \}_{m \in \N}$ is a sequence that converges to  $L(b + \hat\epsilon d).$ This implies
\begin{eqnarray}
L(b + \hat\epsilon d)   &=&  \lim_{m \rightarrow \infty} (\tilde b(h'_m) + \hat\epsilon\tilde d(h'_m) )  \label{eq:sufficient-conditions-vertical-dual-pricing-1} \\
 &=& \limsup \{ \tilde b(h'_m) + \hat\epsilon\tilde d(h'_m) \}_{m \in\N} \label{eq:sufficient-conditions-vertical-dual-pricing-2}\\
&\le& \limsup \{  \tilde b(h'_m)  \}_{m \in M} + \limsup \{ \hat\epsilon\tilde d(h'_m) \}_{m \in\N} \label{eq:sufficient-conditions-vertical-dual-pricing-3}\\
&<& \limsup \{  \tilde b(h'_m)  \}_{m \in M} + \frac{\alpha}{3}.  \label{eq:sufficient-conditions-vertical-dual-pricing-4}
\end{eqnarray}
If  $\limsup \{ \tilde b(h'_m) \}_{m \in \N} <  L(b),$ then by definition of $\alpha,$ 
\begin{align*}
\limsup \{ \tilde b(h'_m) \}_{m \in \N} \le  L(b) - \alpha.
\end{align*}
Then from~\eqref{eq:sufficient-conditions-vertical-dual-pricing-1}-\eqref{eq:sufficient-conditions-vertical-dual-pricing-4}
\begin{eqnarray*}
L(b + \hat\epsilon d) < \limsup \{  \tilde b(h'_m)  \}_{m \in M} + \frac{\alpha}{3} 
\le  L(b) - \alpha +  \frac{\alpha}{3} = L(b) - \frac{2}{3}\alpha
\end{eqnarray*}
which cannot happen since from Step 2,  $L(b + \hat\epsilon d) \ge L(b) - \frac{\alpha}{3} > L(b) - \frac{2}{3} \alpha.$ Therefore $\limsup \{ \tilde b(h'_m) \}_{m \in \N} =  L(b).$  Then by Lemma~\ref{lem:seq-L} there is subsequence of indices $\{ h^{\prime \prime}  \}_{m \in \N}$ from $\{ h^{\prime }  \}_{m \in \N}$ such that 
\begin{align*}
\begin{array}{l}
\tilde b(h''_m) \to L(b) \\
\sum_{k=\ell}^{n} |\tilde{a}^k(h''_m)| \to 0
\end{array}
\end{align*}
and from Claim 3 since  $\{ h^{\prime \prime}  \}_{m \in \N}$ is a subsequence from $\{ h^{\prime }  \}_{m \in \N}$
\begin{align*}
\begin{array}{l}
\tilde b(h''_m) + \hat\epsilon \tilde d(h''_m) \to L(b + \hat \epsilon d) \\
\sum_{k=\ell}^{n} |\tilde{a}^k(h''_m)| \to 0.
\end{array}
\end{align*}

\noindent 5. Claim:
\begin{align*}
\begin{array}{l}
\tilde b(h''_m) + \epsilon \tilde d(h''_m) \to L(b + \epsilon d) \\
\sum_{k=\ell}^{n} |\tilde{a}^k(h''_m)| \to 0.
\end{array}
\end{align*}
holds for all $\epsilon \in [0,\hat \epsilon]$. Proof: this is because for every $\epsilon \in [0, \hat \epsilon]$, $b + \epsilon d$ is a convex combination of the sequences $b$ and $b + \hat \epsilon d$. The claim follows by applying Lemma~\ref{lem:conv-comb} with $b^1 = b$ and $b^2 = b + \hat\epsilon d$. 
\end{proof}

Lemma~\ref{lemma:s-vertical-pricing} is an analogous result for sequences in $I_{3}$ converging to $S(b).$

\begin{lemma}\label{lemma:s-vertical-pricing}
Let $b, d \in \ell_\infty(\left\{0\right\} \cup I)$ such that $-\infty < S(b), S(b+d) < \infty.$ Assume DP.\ref{item:dp-condition-sup-S-vertical} and $\overline{FM}(\ell_\infty(I)) \subseteq \ell_\infty(H)$. Then there exists $\hat \epsilon > 0$ and a sequence $\{h_m\}_{m\in \N} \subseteq I_3$ such that for all $\epsilon \in [0,\hat \epsilon]$:
\begin{align*}
  \tilde d_{\epsilon} (h_m) \to S(b + \epsilon d)
\end{align*}
where $\tilde d_{\epsilon} := \overline{FM}(b + \epsilon d)$. 
\end{lemma}

\begin{proof}
The proof is analogous to Lemma~\ref{lemma:L-vertical-pricing}. Replace  $L$ with $S$, $I_4$ with $I_3$,  and redefine $\alpha$ as 
 \begin{align*}
\alpha :=  S(b) - \sup \{ \limsup \{\tilde{b} (h_{m})\}_{m \in \N} : \{h_{m}\}_{m \in \N} \in \mathcal H_S \} 
\end{align*}
By hypothesis, $-\infty < S(b) < \infty$   so $I_{3}$ is not empty and then by  assumption DP.\ref{item:dp-condition-sup-S-vertical}, $\alpha $  is a positive real number.
 The result follows from  DP.\ref{item:dp-condition-sup-S-vertical}   and noting $\sum_{k=\ell}^{n} |\tilde{a}^k(h_m)| = 0$ for all sequences $\left\{h_m\right\}$ in $I_3.$  
 \end{proof}

\bibliographystyle{plain}
\bibliography{../../references/references}

\begin{thebibliography}{10}

\bibitem{hitchhiker}
C.D. Aliprantis and K.C. Border.
\newblock {\em Infinite Dimensional Analysis: A Hitchhiker's Guide}.
\newblock Springer, second edition, 2006.

\bibitem{anderson-nash}
E.J. Anderson and P.~Nash.
\newblock {\em Linear Programming in Infinite-Dimensional Spaces: {T}heory and
  Applications}.
\newblock Wiley, 1987.

\bibitem{basu2014sufficiency}
A.~Basu, K.~Martin, and C.~T. Ryan.
\newblock On the sufficiency of finite support duals in semi-infinite linear
  programming.
\newblock {\em Operations Research Letters}, 42(1):16--20, 2014.

\bibitem{basu2013projection}
A.~Basu, K.~Martin, and C.~T. Ryan.
\newblock Projection: A unified approach to semi-infinite linear programs and
  duality in convex programming.
\newblock {\em Mathematics of Operations Research}, 40:146--170, 2015.

\bibitem{charnes1963duality}
A.~Charnes, W.W. Cooper, and K.~Kortanek.
\newblock Duality in semi-infinite programs and some works of {H}aar and
  {C}arath\'{e}odory.
\newblock {\em Management Science}, 9(2):209--228, 1963.

\bibitem{charnes1965representations}
A.~Charnes, W.W. Cooper, and K.O. Kortanek.
\newblock On representations of semi-infinite programs which have no duality
  gaps.
\newblock {\em Management Science}, 12(1):113--121, 1965.

\bibitem{duffin-karlovitz65}
R.J. Duffin and L.A. Karlovitz.
\newblock An infinite linear program with a duality gap.
\newblock {\em Management Science}, 12(1):122--134, 1965.

\bibitem{glashoff79}
K.~Glashoff.
\newblock Duality theory of semi-infinite programming.
\newblock In R.~Hettich, editor, {\em Semi-Infinite Programming}, volume~15 of
  {\em Lecture Notes in Control and Information Sciences}, pages 1--16.
  Springer, 1979.

\bibitem{glashoff1983linear}
K.~Glashoff and S.~Gustafson.
\newblock {\em Linear Optimization and Approximation: An Introduction to the
  Theoretical Analysis and Numerical Treatment of Semi-infinite Programs}.
\newblock Springer, 1983.

\bibitem{goberna2007sensitivity}
M.A. Goberna, S.~G{\'o}mez, F.~Guerra, and M.I. Todorov.
\newblock Sensitivity analysis in linear semi-infinite programming: perturbing
  cost and right-hand-side coefficients.
\newblock {\em European Journal of Operational Research}, 181(3):1069--1085,
  2007.

\bibitem{Goberna2010209}
M.A. Goberna, E.~Gonz\'{a}lez, J.E. Mart\'{i}­nez-Legaz, and M.I. Todorov.
\newblock Motzkin decomposition of closed convex sets.
\newblock {\em Journal of Mathematical Analysis and Applications}, 364(1):209
  -- 221, 2010.

\bibitem{goberna1998linear}
M.A. Goberna and M.A. L{\'o}pez.
\newblock {\em Linear Semi-infinite Optimization}.
\newblock Wiley, 1998.

\bibitem{goberna2014post}
M.A. Goberna and M.A. L{\'o}pez.
\newblock {\em Post-Optimal Analysis in Linear Semi-Infinite Optimization}.
\newblock Springer, 2014.

\bibitem{goberna2010sensitivity}
M.A. Goberna, T.~Terlaky, and M.I. Todorov.
\newblock Sensitivity analysis in linear semi-infinite programming via
  partitions.
\newblock {\em Mathematics of Operations Research}, 35(1):14--26, 2010.

\bibitem{haar1924}
A.~Haar.
\newblock Uber lineare ungleichungen.
\newblock {\em Acta Math. Szeged}, 2:1--14, 1924.

\bibitem{hettich1993semi}
R.~Hettich and K.O. Kortanek.
\newblock Semi-infinite programming: theory, methods, and applications.
\newblock {\em SIAM review}, 35(3):380--429, 1993.

\bibitem{holmes}
R.B. Holmes.
\newblock {\em Geometric Functional Analysis and its Applications}.
\newblock Springer, 1975.

\bibitem{karney81}
D.F. Karney.
\newblock Duality gaps in semi-infinite linear programming -- an approximation
  problem.
\newblock {\em Mathematical Programming}, 20(1):129--143, 1981.

\bibitem{karney1982pathological}
D.F. Karney.
\newblock A pathological semi-infinite program verifying {K}arlovitz's
  conjecture.
\newblock {\em Journal of Optimization Theory and Applications},
  38(1):137--141, 1982.

\bibitem{karney85}
D.F. Karney.
\newblock In a semi-infinite program only a countable subset of the constraints
  is essential.
\newblock {\em Journal of Approximation Theory}, 20:129--143, 1985.

\bibitem{martin-stern-ryan}
K.~Martin, {C.T}. Ryan, and {M.} Stern.
\newblock The {S}later conundrum: Duality and pricing in infinite dimensional
  optimization.
\newblock Working paper, Booth School of Business, University of Chicago, 2015.

\bibitem{ponstein1981use}
J.P. Ponstein.
\newblock On the use of purely finitely additive multipliers in mathematical
  programming.
\newblock {\em Journal of Optimization Theory and Applications}, 33(1):37--55,
  1981.

\bibitem{shapiro2009semi}
A.~Shapiro.
\newblock Semi-infinite programming, duality, discretization and optimality
  conditions†.
\newblock {\em Optimization}, 58(2):133--161, 2009.

\end{thebibliography}

\end{document}